\documentclass[11 pt]{amsart}

\usepackage{amsmath,amssymb,amsfonts,graphics}
\usepackage{verbatim}
\usepackage{fullpage}
\newcommand{\N}{\mathbb{N}}
\newcommand{\Z}{\mathbb{Z}}
\newcommand{\R}{\mathbb{R}}

\newcommand{\C}{\mathbb{C}}
\newcommand{\F}{\mathbb{F}}
\newcommand{\G}{\mathbb{G}}

\newcommand{\T}{\mathbb{T}}

\newtheorem{thm}{Theorem}[section]
\newtheorem{cor}[thm]{Corollary}
\newtheorem{lem}[thm]{Lemma}
\newtheorem{prop}[thm]{Proposition}
\theoremstyle{definition}
\newtheorem{defn}[thm]{Definition}

\theoremstyle{remark}
\newtheorem{rem}[thm]{Remark}

\numberwithin{equation}{section}

\begin{document}

\title[Strong Haagerup Inequality]
{Quantum Symmetries and Strong Haagerup Inequalities}

\date{December 22, 2010}

\author{Michael Brannan}

\address{Michael Brannan: Department of Mathematics and Statistics, Queen's University, 99 University Avenue, Kingston, ON CANADA, K7L 3N6.
Email: mbrannan@mast.queensu.ca}


\begin{abstract}
In this paper, we consider families of operators $\{x_r\}_{r \in \Lambda}$ in a tracial C$^\ast$-probability space $(\mathcal A, \varphi)$, whose joint $\ast$-distribution is invariant under free complexification and the action of the hyperoctahedral quantum groups $\{H_n^+\}_{n \in \N}$.  We prove a strong form of Haagerup's inequality for the non-self-adjoint operator algebra $\mathcal B$ generated by $\{x_r\}_{r \in \Lambda}$,
which generalizes the strong Haagerup inequalities for $\ast$-free R-diagonal families obtained by Kemp-Speicher \cite{KeSp}.
As an application of our result, we show that $\mathcal B$ always has the metric approximation property (MAP).  We also apply our techniques to study the reduced C$^\ast$-algebra of the free unitary quantum group $U_n^+$. We show that the non-self-adjoint subalgebra $\mathcal B_n$ generated by the matrix elements of the fundamental corepresentation of $U_n^+$ has the MAP. Additionally, we prove a strong Haagerup inequality for $\mathcal B_n$, which improves on the estimates given by Vergnioux's property RD \cite{Ve}. 
\end{abstract}
\maketitle

\section{Introduction}

Let $\F_n$ denote the free group on $n \le \infty$ generators $g_1, g_2, \ldots, g_n$, and let $C^*_\lambda(\F_n) \subseteq B(\ell^2(\F_n))$ be the C$^\ast$-algebra generated by the left regular representation $\lambda: \F_n \to U(\ell^2(\F_n)).$  Denote by $\ell : \F_n \to \N \cup\{0\}$, the natural reduced word length function associated to the generating set $S = \{g_r, g_r^{-1}\}_{r=1}^n \subset \F_n$.  In 1978, Haagerup published the following result, known as the \textit{Haagerup inequality}:

\begin{thm} \cite[Lemma 1.4]{Ha} \label{thm_Haagerup_ineq}
Let $d \in \N$ and suppose $f \in \ell^2(\F_n)$ is supported on the set $W_d  = \{g \in\F_n: \ \ell(g) = d\}.$ Then $$\|f\|_{\ell^2(\F_n)} \le \|\lambda(f)\|_{C^*_\lambda(\F_n)} \le (d+1) \|f\|_{\ell^2(\F_n)}.$$
\end{thm}

The above inequality was used by Haagerup (together with the fact that the map $\lambda(g) \mapsto e^{-\ell(g)t}\lambda(g)$ defines a unital completely positive map on $C^*_\lambda(\F_n)$, for all $t \ge0$) to show that $C^*_\lambda(\F_n)$ has the metric approximation property \cite{Ha}, even though it is a non-nuclear C$^\ast$-algebra for $n \ge 2$.  Since the publication of this foundational result, the Haagerup inequality has continued to find numerous applications and generalizations in operator algebras, noncommutative harmonic analysis, and geometric group theory.  See \cite{Bo, CoHa, dCHa, Ha, Jo, La1, La2} for example.

One of the key ingredients, implicit in Haagerup's original proof of Theorem \ref{thm_Haagerup_ineq}, is the fact that the generators $\{\lambda(g_r)\}_{r=1}^n$ of $C^*_\lambda(\F_n)$ are algebraically free (and in fact \textit{$\ast$-freely independent} in the sense of Voiculescu's free probability theory).  Using this connection with free independence, Kemp and Speicher \cite{KeSp} showed, using combinatorial  techniques from Voiculescu's free probability theory, that if one restricts to the \textit{non-self-adjoint} operator algebra of convolution operators $\lambda(f) \in C^*_\lambda(\F_n)$ supported on the free \textit{semigroup} $\F_n^+$ generated by $\{g_r\}_{r=1}^n$, then the constants in the Haagerup inequality enjoy a substantial improvement:

\begin{thm} \cite[Theorem 1.4]{KeSp} \label{thm1_KeSp}
For any $d \in \N$ and any $f \in \ell^2(\F_n)$ supported on $W_d\cap \F_n^+$ (i.e. supported on words in $\F_n$ of length $d$ in $g_1, \ldots, g_n$ but \textit{not} their inverses), we have the estimate $$\|f\|_{\ell^2(\F_n)} \le \|\lambda(f)\|_{C^*_\lambda(\F_n)} \le \sqrt{e} \sqrt{d+1} \|f\|_{\ell^2(\F_n)}.$$
\end{thm}

Furthermore, in \cite{KeSp} the authors were able to generalize the above \textit{strong Haagerup inequality} to the much broader context of operator algebras generated by $\ast$-free, identically distributed, R-diagonal operators in a tracial $C^\ast$-probability space $(\mathcal A, \varphi)$.  (Please consult Section \ref{section_prelims} for the relevant definitions.)  Their result can be stated as follows: 
\begin{thm} \cite[Theorem 1.3]{KeSp} \label{thm2_KeSp}
Let $\Lambda$ be an index set, let $\{x_r\}_{r \in \Lambda}$ be a $\ast$-free, identically distributed family of $R$-diagonal operators in a tracial C$^\ast$-probability space $(\mathcal A, \varphi)$, and let $x \in \{x_r\}_{r \in \Lambda}$ be some fixed reference variable.  Then, for any $d \in \N$ and any homogeneous polynomial \begin{eqnarray*}
T = \sum_{i:\{1, \ldots, d\} \to \Lambda} a_i x_{i(1)}x_{i(2)}\ldots x_{i(d)}, && (a_i \in \C),
\end{eqnarray*} of degree $d$ in the variables $\{x_r\}_{r \in \Lambda}$, we have  
$$\|T\|_{L^2(\mathcal A, \varphi)} \le \|T\|_{\mathcal A} \le C_x \sqrt{d} \|T\|_{L^2(\mathcal A, \varphi)},$$ where $C_x \le 515\sqrt{e} \|x\|_\mathcal A^2/\|x\|_{L^2(\mathcal A, \varphi)}^2.$  Moreover, if $x$ has non-negative free cumulants, then $C_x \le \sqrt{e} \|x\|_\mathcal A / \|x\|_{L^2(\mathcal A, \varphi)}$. 
\end{thm}

The remarkable feature of Theorems \ref{thm1_KeSp} and \ref{thm2_KeSp} is the fact that the order of growth of the constants in these inequalities improves from $O(d)$ to $O(\sqrt{d})$.  This slower growth rate is a consequence of two things:  $(1)$ the fact that we are restricting to a \textit{non-self-adjoint} subalgebra of our C$^\ast$-algebra $\mathcal A$, and $(2)$ the fact that the R-diagonal operators under consideration possess a great deal of rotational symmetry in their $\ast$-distributions, which can be exploited.  The basic method of proof in \cite{KeSp} is to approximate the norm $\|T\|_{\mathcal A}$ of a given homogeneous polynomial with the associated noncommutative $L^p$-norms $\|T\|_{L^p(\mathcal A, \varphi)}$, ($p \to \infty$).  For $p \in 2\N$ (an even integer), this amounts to the calculation of certain joint moments, which can then be expressed in terms of Speicher's free cumulants, and estimated efficiently using the freeness and R-diagonality assumptions.  We remark here that de la Salle \cite{dS} has recently considered the framework of Theorem \ref{thm2_KeSp} in the category of operator spaces, and obtained strong Haagerup inequalities with \textit{operator coefficients}.  

On a different note, several deep connections between free probability and certain classes of compact quantum groups have recently emerged, particularly in the study of \textit{quantum symmetries} of families of random variables.  Perhaps most illustrative of this connection is the \textit{free de Finetti theorem} of K\"ostler-Speicher \cite{KoSp} (and its generalizations \cite{BaCuSp}), which says that an infinite sequence of random variables in a W$^\ast$-probability space is conditionally free if and only if its joint $\ast$-distribution is invariant under the action of the quantum permutation groups $\{S_n^+\}_{n \in \N}$.  Another example illustrating this connection is the asymptotic freeness of the standard generators of many combinatorial quantum groups such as $S_n^+, H_n^+, O_n^+$, $U_n^+$ \cite{BaCo, BaCuSp}.       

In this paper, we prove a generalization of the strong Haagerup inequality of Kemp-Speicher (Theorem \ref{thm2_KeSp}).  Continuing with the above theme of connecting free probability and compact quantum groups, we consider families of operators $\{x_r\}_{r \in \Lambda}$ in a C$^\ast$-probability space $(\mathcal A, \varphi)$, which are not necessarily $\ast$-free, but instead possess certain quantum symmetries in their joint $\ast$-distributions.  The setup for our result is as follows:  for each $n \in \N$, let $H_n^+$ denote the \textit{hyperoctahedral quantum group} of dimension $n$ \cite{BaBiCo}.  We say that a family of random variables $\{x_r\}_{r \in \Lambda}$ in a C$^\ast$-probability space $(\mathcal A, \varphi)$ has an $H^+$-invariant joint $\ast$-distribution if the joint $\ast$-distribution each sub-$n$-tuple $\{x_{r(l)}\}_{l=1}^n$ of $\{x_{r}\}_{r \in \Lambda}$ is invariant under the natural action of $H_n^+$.  We say that the joint $\ast$-distribution of $\{x_r\}_{r \in \Lambda}$ is invariant under free complexification if $\{zx_r\}_{r \in \Lambda}$ has the same joint $\ast$-distribution as $\{x_r\}_{r \in \Lambda}$, for any Haar unitary $z \in (\mathcal A, \varphi)$  which is $\ast$-free from $\{x_r\}_{r \in \Lambda}$.  Our main theorem is the following:

\begin{thm} \label{thm_main_SHI_intro}
Let $(\mathcal A, \varphi)$ be a tracial C$^\ast$-probability space, and let $\{x_r\}_{r \in \Lambda} \subset (\mathcal A, \varphi)$ be a family of random variables.  Suppose that the joint $\ast$-distribution of $\{x_r\}_{r \in \Lambda}$ is $H^+$-invariant and invariant under free complexification.  Let $x \in \{x_r\}_{r \in \Lambda}$ be a fixed reference variable.  Then for any homogeneous polynomial \begin{eqnarray*}T = \sum_{i:[d] \to \Lambda} a_i x_{i(1)}x_{i(2)}\ldots x_{i(d)}, && (a_i \in \C),
\end{eqnarray*} of degree $d$ in the variables $\{x_r\}_{r \in \Lambda}$ and any $p \in 2\N \cup \{\infty\}$, we have $$\|T\|_{L^2(\mathcal A, \varphi)} \le \|T\|_{L^p(\mathcal A, \varphi)} \le 4^5 \cdot (3e)^2\sqrt{e}
\frac{\|x\|_{L^p(\mathcal A, \varphi)}^2}{\|x\|_{L^2(\mathcal A, \varphi)}^2} \sqrt{d+1} \|T\|_{L^2(\mathcal A, \varphi)}.$$
\end{thm}
Using Theorem \ref{thm_main_SHI_intro}, we study the structure of the non-self-adjoint operator algebra $\mathcal B$ generated by the family $\{x_r\}_{r \in \Lambda}$.  In particular, we show that $\mathcal B$ always has the metric approximation property.  We also indicate how any $\ast$-free, identically distributed R-diagonal family $\{x_r\}_{r \in \Lambda} \subset (\mathcal A, \varphi)$ satisfies the hypotheses of Theorem \ref{thm_main_SHI_intro}.  On the other hand, we show that there are many natural families of random variables (coming from certain free complexified quantum groups) satisfying the hypotheses of Theorem \ref{thm_main_SHI_intro}, which are \textit{not} $\ast$-free.  Using a modification of Theorem \ref{thm_main_SHI_intro} for \textit{bi-invariant arrays} of random variables, we obtain a strong Haagerup inequality for Wang's free unitary quantum group $U_n^+$ (\cite{Wa}), improving on the Haagerup inequalities for $U_n^+$ obtained by Vergnioux \cite{Ve}.  Viewing $U_n^+$ as the non-cocommutative analogue of the compact quantum group associated to $C^*_\lambda(\F_n)$, our result is the non-cocommutative analogue of the Kemp-Speicher strong Haagerup inequality (Theorem \ref{thm1_KeSp}) for $\F_n$.

The remainder of the paper is organized as follows:  Section \ref{section_prelims} contains all of the notation and basic facts we will need from free probability and compact quantum groups.  In Section \ref{section_main_thm}, we restate and prove Theorem \ref{thm_main_SHI_intro}.  In Section \ref{section_bi_invariant_arrays}, we generalize Theorem \ref{thm_main_SHI_intro} to arrays $\{x_{rs}\}_{1 \le r,s \le n}$ of random variables, whose joint $\ast$-distribution is $H^+_n$-bi-invariant (Theorem \ref{main_theorem_arrays}).  In Section \ref{section_MAP}, we study the unital (norm closed) non-self-adjoint operator algebras $\mathcal B \subset (\mathcal A, \varphi)$ generated by families $\{x_r\}_{r \in\Lambda} \subset (\mathcal A, \varphi)$ satisfying the hypotheses of Theorem \ref{thm_main_SHI_intro} or \ref{main_theorem_arrays}.  We show that the natural complex Ornstein-Uhlenbeck type semigroup $\{\Gamma_t\}_{t > 0}$ acting on $\mathcal B$ by $\Gamma_t(x_{i(1)}x_{i(2)}\ldots x_{i(d)}) = e^{-dt}x_{i(1)}x_{i(2)}\ldots x_{i(d)}$, is completely contractive for all $t \ge 0$.  Using this fact, together with the norm estimates provided by our Haagerup inequalities, we prove that $\mathcal B$ has the metric approximation property.  In Section \ref{quantum_groups}, we consider the free unitary quantum groups $U_n^+$ ($n \in \N$), and show that the strong Haagerup inequalities and metric approximation properties obtained in Sections \ref{section_main_thm} - \ref{section_MAP} apply to the non-self-adjoint operator algebra $\mathcal B_n \subset L^\infty(U_n^+)$, generated by the matrix elements of the fundamental corepresentation of $U_n^+$.   This yields Theorem \ref{thm_free_unitary}, which is our non-cocommutative analogue of Theorem \ref{thm1_KeSp}.  We also discuss a general method for obtaining more non-trivial examples of families which satisfy the hypotheses of our theorems.  

\subsection*{Acknowledgements}
It is a pleasure to thank my doctoral supervisors James A. Mingo and Roland Speicher for many fruitful discussions and for their continued guidance while working on this project.  This research was partially supported by an NSERC Canada Graduate Scholarship.   

\section{Preliminaries and Notation} \label{section_prelims}

We begin by briefly reviewing the relevant facts from free probability and C$^\ast$-algebraic compact matrix quantum groups that will be needed in this paper.  Our main reference for free probability will be the monograph \cite{NiSp}.  For the basics of compact quantum groups we refer to the textbook \cite{Ti} and the foundational paper of Woronowicz \cite{Wo}. 

\subsection{Noncommutative Probability Spaces and Free Independence}

\begin{defn} \label{defn_NCPS}
\begin{enumerate} \item A \textit{noncommutative probability space (NCPS)} is a pair $(\mathcal A, \varphi)$, where $\mathcal A$ is a unital $\ast$-algebra over $\C$, and $\varphi:\mathcal A \to \C$ is a state (i.e. a linear functional such that $\varphi(1_\mathcal A) = 1$ and $\varphi(aa^*) \ge 0$ for all $a \in \mathcal A$).  We say that $(\mathcal A, \varphi)$ is \textit{tracial} if $\varphi$ is a trace on $\mathcal A$. 
\item A \textit{C$^*$-probability space} is a NCPS $(\mathcal A, \varphi)$ where $\mathcal A$ is a unital C$^\ast$-algebra, and $\varphi$ is a \textit{faithful state} (i.e. $\varphi(aa^*) = 0 \Leftrightarrow a = 0$). 
\item  If $(\mathcal A, \varphi)$ is a NCPS, we call elements of $\mathcal A$ \textit{random variables}.  A random variable $x \in (\mathcal A, \varphi)$ is called \textit{centered} if $x \in \ker \varphi$. 
\item  Let $X = \{x_r\}_{r \in \Lambda}$ be a family of random variables in a NCPS $(\mathcal A, \varphi)$, and let $\C\langle t_r, t_r^*: \ r \in \Lambda \rangle$ denote the $\ast$-algebra of noncommutative polynomials in the indeterminates $\{t_r\}_{r \in \Lambda}$.  Denote by \begin{eqnarray*} ev_X: \C\langle t_r, t_r^*: \ r \in \Lambda \rangle \to \mathcal A, && ev_X(t_r) = x_r \end{eqnarray*} the canonical evaluation $\ast$-homomorphism determined by the family $X$.  The \textit{joint $\ast$-distribution} of the family $X = \{x_r\}_{r \in \Lambda}$ is the linear functional $\varphi_X: \C\langle t_r, t_r^*:  r \in \Lambda \rangle \to \C$ given by 
\begin{eqnarray*}
\varphi_X(p) = \varphi(ev_X(p)), && (p \in \C\langle t_r, t_r^*:  r \in \Lambda \rangle).
\end{eqnarray*}
\item Let $(\mathcal A_j, \varphi_j)$ ($j=1,2$) be two NCPS's, and consider two families of random variables $X = \{x_r\}_{r \in \Lambda} \subset (\mathcal A_1, \varphi_1)$ and $Y = \{y_r\}_{r \in \Lambda} \subset (\mathcal A_2, \varphi_2)$.  We say that $X$ and $Y$ are \textit{identically distributed} if they have the same joint $\ast$-distributions:  $\varphi_X = \varphi_Y$. 
\item If $(\mathcal A, \varphi)$ is a tracial C$^\ast$-probability space, and $p \in [1, \infty)$ we denote by $L^p(\mathcal A, \varphi)$ the associated \textit{noncommutative $L^p$-space}, which is the completion of $\mathcal A$ with respect to the norm $\|x\|_{L^p(\mathcal A, \varphi)} = \varphi((xx^*)^{p/2})^{1/p}$.  For $p = \infty$, we identify $L^\infty(\mathcal A, \varphi)$ with $\mathcal A$, and note that $$\|x\|_{L^\infty(\mathcal A, \varphi)} = \lim_{p \to \infty} \|x\|_{L^p(\mathcal A, \varphi)},$$ since $\varphi$ is faithful. 
\end{enumerate} \end{defn}

We now recall the definition of free independence for noncommutative probability spaces.

\begin{defn} \label{defn_freeness} \begin{enumerate}
\item Let $(\mathcal A, \varphi)$ be a NCPS and let $\{\mathcal A_r\}_{r \in \Lambda}$ be a family of unital subalgebras of $\mathcal A$.  The family $\{\mathcal A_r\}_{r \in \Lambda}$ is said to be
\textit{freely independent with respect to $\varphi$} (or just
\textit{free} if the state $\varphi$ is understood), if the following condition holds: for any choice of indices $r(1) \ne r(2), r(2) \ne r(3), \ldots, r(k-1) \ne r(k) \in \Lambda$ and any choice of centered random variables variables $x_{r(j)} \in \mathcal
A_{r(j)} \cap \ker \varphi$, we have the equality $$\varphi(x_{r(1)}x_{r(2)} \ldots x_{r(k)}) = 0.$$  If each $\mathcal A_r$ is a $\ast$-subalgebra, the we say that $\{\mathcal A_r\}_{r \in \Lambda}$ are \textit{$\ast$-free}. 
\item A family of
noncommutative random variables $\{x_r\}_{r \in \Lambda} \subseteq
(\mathcal A, \varphi)$ is  \textit{free} if the family of
unital subalgebras \begin{eqnarray*}
\{\mathcal A_r\}_{r \in \Lambda}, && \mathcal A_r = \textrm{Alg}\langle 1_\mathcal A,
x_r\rangle, \end{eqnarray*}
 are free in the sense of $(1)$.  We say that $\{x_r\}_{r \in \Lambda}$ are \textit{$\ast$-free} if we replace the subalgebras $\mathcal A_r = \textrm{Alg}\langle 1_\mathcal A,
x_r\rangle$ above by the $\ast$-subalgebras $\textrm{Alg}\langle 1_\mathcal A,x_r, x_r^*\rangle$.
\end{enumerate}
\end{defn}

\subsection{Non-Crossing Partitions and Free Cumulants} \label{section_partitions}

An important tool for studying the notion of free independence and for our calculations will be the free cumulant functionals, $\{\kappa_n:\mathcal A^n \to \C\}_{n \in \N}$, associated to any NCPS $(\mathcal A, \varphi)$.  These were first introduced by Speicher in \cite{Sp}.   

Throughout this section and the rest of the paper, we will frequently be dealing with partitions of the ordered sets $\{1,2, \ldots, k \}$ and multi-indices $i = (i(1), \ldots, i(k)) \in \Lambda^k$ where $\Lambda$ is any index set and $k \in \N$.  For ease of notation, we will denote the ordered set $\{1, \ldots, k\}$ by $[k]$, and interchangeably view multi-indices $$i = (i(1), \ldots , i(k)) \in \Lambda^k,$$ as functions 
\begin{eqnarray*} i:[k] \to \Lambda, && j \mapsto i(j). 
\end{eqnarray*}  Also, given $r$ functions $i_1, \ldots, i_r:[k] \to \Lambda$, we will often regard the $r$-tuple $I = (i_1, \ldots, i_r)$ as the function $I:[rk] \to \Lambda$ in the obvious way.

\begin{defn}\label{defn_partition}
\begin{enumerate}
\item A \textit{partition} $\pi$ of the set $[k]$ is a collection of
disjoint, non-empty subsets $V_1, \ldots, V_r \subseteq [k]$ such that $V_1 \cup
\ldots \cup V_r = [k]$. $V_1, \ldots, V_r$ are called the
\textit{blocks} of $\pi$, and we set $|\pi| = r$ - the number of blocks of $\pi$. If $s,t \in [k]$ are
in the same block of $\pi$, we write $s \sim_\pi t$. The collection
of partitions of $[k]$ is denoted by $\mathcal P(k)$. 
\item Given $\pi, \sigma \in \mathcal P(k)$, we say that $\pi \le
\sigma$ if each block of $\pi$ is contained in a block of $\sigma$.
There is a least element of $\mathcal P(k)$ which is larger than
both $\pi$ and $\sigma$, which we denote by $ \pi \vee \sigma$.  There is also a maximal element $\pi \wedge \sigma \in \mathcal P(k)$ which is smaller than both $\pi$ and $\sigma$. With these operations, $\mathcal P(k)$ is a finite lattice.  
\item Given a function $i:[k] \to \Lambda$, we denote by $\ker i$ the element of
$\mathcal P(k)$ whose blocks are the equivalence classes of the
relation $$s \sim_{\ker i} t \iff i(s) = i(t).$$ Note that
if $\pi  \in \mathcal P(k)$, then $\pi \le \ker i$ is equivalent to
the condition that whenever $s$ and $t$ are
in the same block of $\pi$, $i(s)$ must equal $i(t)$ (i.e the
 function $i:[k] \to \Lambda$ is constant on the blocks of $\pi$).  
\item We say that $\pi \in \mathcal P(k)$ is
\textit{non-crossing} if whenever $V, W$ are blocks of $\pi$ and $s_1
< t_1 < s_2 < t_2$ are such that $s_1, s_2 \in V$ and $t_1, t_2 \in
W$, then $V = W$. We can also define non-crossing partitions recursively: a partition $\pi \in  \mathcal P(k)$ is non-crossing
if and only if it has a block $V$ which is an interval, such that
$\pi \backslash V $ is a non-crossing partition of $[k] \backslash V \cong [k- |V|]$.
The set of non-crossing partitions of $[k]$ is denoted by $NC(k)$. 
\item Given $\pi, \sigma \in NC(k)$, we define $\pi \vee_{NC} \sigma \in NC(k)$ to be the least element of $NC(k)$ which is larger than both $\pi$ and $\sigma$.  With the operation $\le$ and $\wedge$ induced by the inclusion $NC(k) \subset \mathcal P(k)$, and the operation $\vee_{NC}$, $NC(k)$ becomes a finite lattice itself.  
\item  A partition $\pi \in \mathcal P(k)$ is called a
\textit{pairing} if every block of $\pi$ contains exactly $2$
elements.  The set of all pairings of $[k]$ is denoted by $\mathcal
P_2(k)$.  We call a partition $\pi \in NC(k) \cap \mathcal P_2(k)$
a \textit{non-crossing pairing}, and denote this set of pairings by
$NC_2(k)$.  A partition $\pi \in \mathcal P(k)$ is called \textit{even} if every block of $\pi$ has even cardinality.  We denote the set of even partitions (resp. even non-crossing partitions) of $[k]$ by $\mathcal P_e(k)$ (resp. $NC_e(k)$).  Note that $k$ must be even for $\mathcal P_e(k)$ to be non-empty. 
\item  Given a function $\epsilon:[k] \to \{1, \ast\}$, we define the set $\mathcal P^\epsilon(k)$ of
\textit{$\epsilon$-partitions of $[k]$} to
be the set of all even partitions $\pi \in \mathcal P_e(k)$ with the property that on each block $$V = \{j_1 < j_2 < \ldots <j_{2r-1} < j_{2r}\},$$ of $\pi$, $\epsilon|_V$ is an alternating function.  I.e.
\begin{eqnarray*}
\epsilon|_V = \underbrace{(1,\ast, \ldots, 1,\ast)}_{2r \textrm{ times}},  &\textrm{or}& \epsilon|_V = \underbrace{(\ast, 1,  \ldots, \ast, 1)}_{2r \textrm{ times}}, \end{eqnarray*} for all blocks $V$ of $\pi$.  We also put $NC^{\epsilon}(k) := NC(k) \cap P^{\epsilon}(k).$
\end{enumerate}
\end{defn}

Since $NC(k)$ is a lattice, it has a M\"obius function $\mu_k:NC(k)\times NC(k) \to \Z$, which is  is well-known.  The first major study of the lattice structure of $NC(k)$ was done by Kreweras in 1972.  We refer to Kreweras' original paper \cite{Kr} and \cite[Chapters 9-10]{NiSp} for more details on this lattice structure.  For our purposes, we will only need the fact that  
\begin{eqnarray*}
\mu_k(0_k,1_k) = (-1)^{k-1}C_{k-1},
\end{eqnarray*} and that for any $\sigma \in NC(k)$ 
\begin{eqnarray*}
|\mu_k(\sigma,1_k)| \le |\mu_k(0_k, 1_k)| \le C_{k-1}.
\end{eqnarray*} where $0_k, 1_k$ are respectively the smallest and biggest elements of $NC(k)$, and $C_k$ denotes the $k$th Catalan number
\begin{eqnarray*}
C_k = |NC(k)| =  \frac{1}{k}\left(
\begin{array}{c} 2k \\ k-1
\end{array}\right) \le 4^k.
\end{eqnarray*}

Let $\mathcal A$ be a $\C$-vector space equipped with a sequence of multilinear functionals $\{\psi_n:\mathcal A^n \to \C\}_{n \in \N}$.  For each $k \in \N$ and $\pi \in NC(k)$, there is a canonical way to define a new $k$-linear functional $\psi_\pi:\mathcal A^k \to \C$ from the family $\{\psi_n\}_n$ as follows:  if $\{V_j\}_{j=1}^r$ denote the blocks of $\pi$, where $V_j = \{i_j(1) < i_j(2) < \ldots < i_j(|V_j|)\}$, then $\psi_\pi$ is defined by the equation \begin{eqnarray*}
\psi_\pi[x_1, \ldots, x_k] = \prod_{j=1}^r \psi_{|V_j|}[x_{i_j(1)}, \ldots, x_{i_j(|V_j|)}], && ((x_1, \ldots x_k) \in \mathcal A^k).
\end{eqnarray*} 
We are now ready to define the \textit{free cumulants} associated to a NCPS $(\mathcal A, \varphi)$.

\begin{defn} \label{defn_free_cums}
\begin{enumerate}
\item Let $(\mathcal A, \varphi)$ be a NCPS.  The \textit{free cumulant functionals} of $(\mathcal A, \varphi)$ are the family of multilinear functionals $\{\kappa_n:\mathcal A^n \to \C\}_{n \in \N}$ defined recursively by the equations \begin{equation} \label{eqn_free_cum_implicit}\varphi(x_1x_2 \ldots  x_k) = \sum_{\pi \in NC(k)} \kappa_\pi[x_1, x_2, \ldots, x_k], \ \ \ \ ((x_1, \ldots, x_k) \in \mathcal A^k).
\end{equation}
\item Given a random variable $x \in (\mathcal A, \varphi)$, the \textit{free cumulants of $x$} are simply the collection of all quantities 
\begin{eqnarray*}
\kappa_k[x^{\epsilon(1)}, \ldots, x^{\epsilon(k)}], && (\epsilon:[k] \to \{1,\ast\}, \ k \in \N).
\end{eqnarray*}
\end{enumerate}
\end{defn}
Alternatively, by considering the family of moment functionals $\{\varphi_k:\mathcal A^k \to \C\}_{k \in \N}$ given by \begin{eqnarray*}
\varphi_k[x_1, x_2, \ldots, x_k] = \varphi(x_1x_2\ldots x_k), && ((x_1, \ldots, x_n) \in \mathcal A^k),
\end{eqnarray*}    
and applying M\"obius inversion to (\ref{eqn_free_cum_implicit}), we have \begin{equation} \label{eqn_cums_explicit}
\kappa_k[x_1, \ldots, x_k] = \sum_{\sigma \in NC(k)} \mu_k(\sigma, 1_k) \varphi_\sigma[x_1, \ldots,x_k], \ \ \ \ ((x_1, \ldots, x_k) \in \mathcal A^k).\end{equation}
More generally, we have \begin{equation} \label{eqn_free_cum_implicit_full}
\kappa_\pi[x_1, \ldots, x_k] = \sum_{\substack{\sigma \in NC(k) \\ \sigma \le \pi}} \mu_k(\sigma, \pi) \varphi_\sigma[x_1, \ldots,x_k], \ \ \ \ (\pi \in NC(k)).\end{equation}

The following result, originally proved by Speicher \cite{Sp} (see also \cite[Theorem 11.16]{NiSp}), characterizes freeness in terms of the vanishing of mixed free cumulants.

\begin{thm}\label{thm_freeness_vanishing_mixed_cums}
Let $(\mathcal A, \varphi)$ be a NCPS with free cumulant functionals $\{\kappa_n\}_{n \in \N}$  A family of  unital subalgebras $\{\mathcal A_r\}_{r \in \Lambda}$ is free if and only if for any $k \ge 2$ and any function $i:[k] \to \Lambda$, we have 
\begin{eqnarray*}
\kappa_k|_{\mathcal A_{i(1)} \times \mathcal A_{i(2)} \times \ldots \times \mathcal A_{i(k)}} = 0, && (\textrm{unless $i(1) = i(2) = \ldots = i(k)$}). 
\end{eqnarray*} 
\end{thm}

\subsection{R-Diagonal Elements and Free Complexification}

R-diagonal elements are defined in terms of the structure of their free cumulants.

\begin{defn} \label{defn_R_diagonal}
A random variable $x \in (\mathcal A, \varphi)$ is called an \textit{R-diagonal element} if the only non-zero free cumulants of $x$ are the alternating even ones.  I.e. \begin{eqnarray*} 
\kappa_{2k+1}[x^{\epsilon(1)}, \ldots, x^{\epsilon(2k + 1)}] = 0, &&  ( \forall k \in \N, \ \epsilon:[2k+1] \to \{1, \ast\}),
\end{eqnarray*}   
and 
\begin{eqnarray*} 
\kappa_{2k}[x^{\epsilon(1)}, \ldots, x^{\epsilon(2k)}] \ne 0,
\end{eqnarray*} 
only if $\epsilon =(1,\ast, \ldots, 1,\ast)$, or $\epsilon =(\ast, 1,  \ldots, \ast, 1)$.
\end{defn}

The standard examples of R-diagonal elements are \textit{Haar unitary} random variables and \textit{circular}  random variables.  A random variable $z \in (\mathcal A,\varphi)$ is Haar unitary if $z$ is unitary, and $\varphi(z^n) = \delta_{n,0}$ for all $n \in \Z$.  From these equations it follows that  the free cumulants of a Haar unitary $z$ satisfy 
\begin{eqnarray*}
\kappa_{2k}[z, z^*,\ldots,z,z^*] = \kappa_{2k}[z^*,z,\ldots, z,z^*] = (-1)^{k-1} C_{k-1}, 
\end{eqnarray*}  
and are identically zero otherwise (\cite[Proposition 15.1]{NiSp}).  A random variable $c \in (\mathcal A, \varphi)$ is (standard) circular if and only if its free cumulants satisfy $\kappa_2[c,c^*] = \kappa_2[c^*,c] = 1$, and are zero otherwise.  The standard unitary generators $\lambda(g_1), \ldots, \lambda(g_k)$ of the reduced group $C^*$-algebra $C^*_\lambda(\F_k)$ are Haar unitary, and $\ast$-free with respect to the canonical trace-state $\tau:C^*_\lambda(\F_k) \to \C$.  If $s_1,s_2 \in (\mathcal A, \varphi)$  are two free standard semicircular random variables (i.e. $s_1$ and $s_2$ are free, identically distributed, and self-adjoint with spectral measures $d\mu_{s_i}(x) = \frac{1}{2\pi} \sqrt{4-x^2} \chi_{[-2,2]}dx$), then $$c= \frac{s_1 +is_2}{\sqrt{2}},$$ is a circular random variable.  In Voiculescu's free probability theory, the semicircular variable is the ``free analogue'' of a real Gaussian random variable.  The circular variable $c$ therefore can be thought of as the free analogue of a complex Gaussian random variable.

We now introduce the important notion of free complexification, which can be thought of as the free analogue of multiplying a classical complex random vector by a randomly chosen phase factor $e^{i\theta} \in \T$.

\begin{defn} \label{defn_free_complexification}
Let $\{x_r\}_{r \in \Lambda}$ be a family random variables in a NCPS $(\mathcal A, \varphi)$, and let $z \in (\mathcal A, \varphi)$ be a Haar unitary which is $\ast$-free from $\{x_r\}_{r \in \Lambda}$.  We call the family $\{w_r\}_{r \in \Lambda}$, where $w_r = zx_r$, the \textit{free complexification} of $\{x_r\}_{r \in \Lambda}$.  We say that the joint $\ast$-distribution of $\{x_r\}_{r \in \Lambda}$ is \textit{invariant under free complexification} if $\{x_r\}_{r \in \Lambda}$ and $\{w_r\}_{r \in \Lambda}$ are identically distributed. 
\end{defn}

\begin{rem} \label{rem_free_compl}
If $z^\prime \in (\mathcal A, \varphi)$ is another Haar unitary which is $\ast$-free from $\{x_r\}_{r \in \Lambda}$ and $w_r^\prime =  z^\prime x_i$, then $\{w_r\}_{r \in \Lambda}$ and $\{w_r^\prime\}_{r \in \Lambda}$ are identically distributed families.  Therefore, the notion of free complexification as an operation on joint $\ast$-distributions, is well defined.  We also record here the elementary fact if $\{w_r\}_{r \in \Lambda}$ denotes the free complexification of $\{x_r\}_{r \in \Lambda}$, then the joint $\ast$-distribution of  $\{w_r\}_{r \in \Lambda}$ is invariant under free complexification.
\end{rem} 

Free complexification and R-diagonality are intimately related through the following theorem.

\begin{thm} \label{thm_R_diagonality_free_compl}
$(i)$ Let $z, x \in (\mathcal A, \varphi)$ be random variables, and suppose that $z$ is Haar unitary and $\ast$-free from $x$.  Then $x$ is R-diagonal if and only if $x$ and $zx$ are identically distributed.  In particular if $x$ is any random variable and $z$ is Haar unitary and $\ast$-free from $x$, then $zx$ is R-diagonal. \\
$(ii)$ More generally,  suppose that $\{x_r\}_{r \in \Lambda} \subset (\mathcal A, \varphi)$ is a $\ast$-free family of even elements (i.e. all odd $\ast$-moments of each $x_r$ are zero), and $z \in (\mathcal A, \varphi)$ is a Haar unitary which is $\ast$-free from $\{x_r\}_{r \in \Lambda}$, then the free complexification $\{zx_r\}_{r \in \Lambda}$ is a $\ast$-free family of R-diagonal elements.  \\
\end{thm}

\begin{proof}
Statement $(i)$ is \cite[Corollary 15.9 and Theorem 15.10]{NiSp}.  Statement $(ii)$ is a special case of \cite[Theorem 1.13]{NiSp2}. 
\end{proof}

\subsection{Quantum Groups and Invariant Distributions}

\begin{defn} \label{defn_cmqg}
A \textit{compact matrix quantum group (CMQG)} is a pair $\G = (A,U)$,
where $A$ is a unital C$^\ast$-algebra, and $U = [u_{rs}]_{1 \le r,s
\le n} \in M_n(A)$ is a unitary element satisfying the following three conditions:
\begin{enumerate}
\item $A$ is generated as a $C^\ast$-algebra by the set $\{ u_{rs}:
\ 1 \le r,s \le n \}$;
\item The matrix $\overline{U} := [u_{rs}^*]_{1 \le r,s \le n}$ is invertible in $M_n(A)$; 
\item  There exists a $\ast$-homomorphism $\Delta:A \to A \otimes A$
 such that $$\Delta u_{rs} = \sum_{k=1}^n u_{rk}\otimes u_{ks}, \ \ \ \ (1 \le r,s \le n).$$
\end{enumerate}
\end{defn} 
By definition, the $\ast$-homomorphism $\Delta$ (called the \textit{coproduct} of $A$) satisfies the coassociativity law $$(\Delta \otimes id) \circ \Delta  = (id \otimes \Delta) \circ \Delta.$$ 

A $k$-dimensional \textit{unitary corepresentation} of $\G = (A, U)$ is a unitary matrix $V  = [v_{rs}]_{1 \le r,s \le k} \in M_k(A)$ such that  
\begin{eqnarray*}
\Delta v_{rs} = \sum_{l=1}^k v_{rl} \otimes v_{ls}, && (1 \le r,s \le k).
\end{eqnarray*}
By definition, the matrix $U$ defining $\G$ is a corepresentation, called the \textit{fundamental corepresentation} of $\G$.  The unit $1_A$ is also a corepresentation of $\G$, called the \textit{trivial corepresentation}.

Recall that for any CMQG $\G = (A, U)$, there exists a unique state $h:A \to \C$, called the \textit{Haar state}, which is bi-invariant with respect to the coproduct $\Delta$.  I.e. 
\begin{eqnarray} \label{eqn_Haar}
(h \otimes id ) \circ \Delta a = ( id \otimes h ) \circ \Delta a = h(a)1, && (a \in A).
\end{eqnarray}  
The existence and uniqueness of the Haar state was shown by Woronowicz in \cite{Wo}.  When $h$ is a trace (as it will always be for us) and $p \in [1,\infty]$, we denote by $L^p(\G):=L^p(A, h)$ the noncommutative $L^p$-space associated to the GNS representation of the Haar state $h$. 

\begin{defn}  \label{defn_inv_dist}
Let $\G = (A, U)$ be a CMQG with fundamental corepresentation $U = [u_{rs}]_{1 \le r,s \le n} \in M_n(A)$, and let $(\mathcal A, \varphi)$ be a NCPS. \\
$(i).$  We say that an $n$-tuple of random variables $X = \{x_{r}\}_{r=1}^n \subset (\mathcal A, \varphi)$ has a \textit{(left) $\G$-invariant joint $\ast$-distribution} if the family $Y = \{y_r\}_{r=1}^n \subset A \otimes \mathcal A$ defined by \begin{eqnarray*} y_r = \sum_{k=1}^n u_{rk}\otimes x_k, && (1 \le r \le n),
\end{eqnarray*}  
has the same joint $\ast$-distribution (with respect to $id \otimes \varphi$) as $X$.  That is, for any $\ast$-polynomial $p \in \C\langle t_r,t_r^*: 1 \le r \le n \rangle$, $$(id \otimes \varphi)(ev_Y(p)) = \varphi_X(p)1_A.$$
$(ii).$ We say that an $n \times n$ array of random variables $X = \{x_{rs}\}_{1 \le r,s \le n} \subset (\mathcal A, \varphi)$ has a \textit{left-$\G$-invariant joint $\ast$-distribution} if the family $Y = \{y_{rs}\}_{1 \le r,s \le n} \subset A \otimes \mathcal A$  defined by \begin{eqnarray*} y_{rs} = \sum_{k=1}^n u_{rk}\otimes x_{ks}, && (1 \le r,s \le n),
\end{eqnarray*}  
has the same joint $\ast$-distribution (with respect to $id \otimes \varphi$) as $X$.  Similarly, $X$ has a \textit{right-$\G$-invariant joint $\ast$-distribution} if the family $Z = \{z_{rs}\}_{1 \le r,s \le n} \subset \mathcal  A \otimes A$  defined by \begin{eqnarray*} z_{rs} = \sum_{k=1}^n x_{rk}\otimes u_{ks}, && (1 \le r,s \le n),
\end{eqnarray*}  
has the same joint $\ast$-distribution (with respect to $\varphi \otimes id$) as $X$.  We say that $X$ has a \textit{$\G$-bi-invariant joint $\ast$-distribution} if the $\ast$-distribution is both left and right $\G$-invariant.  I.e., for any $\ast$-polynomial $p \in \C\langle t_{rs},t_{rs}^*: 1 \le r,s \le n \rangle$, $$(id \otimes \varphi)(ev_Y(p)) = (\varphi \otimes id)(ev_Z(p)) = \varphi_X(p)1_A.$$ 
\end{defn}

\begin{rem} \label{rem_on _invariance}
In other words, an $n$-tuple $X = \{x_i\}_{i=1}^n \subset (\mathcal A, \varphi)$ has a $\G$-invariant joint $\ast$-distribution if and only if for all $k \in \N$, $i:[k] \to [n]$, and $\epsilon:[k] \to \{1,\ast\},$ we have the algebraic identity
$$\varphi\big(x_{i(1)}^{\epsilon(1)}\ldots
x_{i(k)}^{\epsilon(k)}\big)1_A = \sum_{j:[k] \to [n]}
u_{i(1)j(1)}^{\epsilon(1)}\ldots
u_{i(k)j(k)}^{\epsilon(k)}\varphi\big(x_{j(1)}^{\epsilon(1)}\ldots
x_{j(k)}^{\epsilon(k)}\big). $$
Similarly, $X = \{x_{rs}\}_{1 \le r,s \le n}$ has a $\G$-bi-invariant joint $\ast$-distribution if and only if for all $k \in \N$, $i, j:[k] \to [n]$, and $\epsilon:[k] \to \{1,\ast\},$ we have the algebraic identities
\begin{eqnarray*}
&& \varphi\big(x_{i(1)j(1)}^{\epsilon(1)}\ldots
x_{i(k)j(k)}^{\epsilon(k)}\big)1_A \\ &=& \sum_{l:[k] \to [n]}
u_{i(1)l(1)}^{\epsilon(1)}\ldots
u_{i(k)l(k)}^{\epsilon(k)}\varphi\big(x_{l(1)j(1)}^{\epsilon(1)}\ldots
x_{l(k)j(k)}^{\epsilon(k)}\big) \\ 
&=& \sum_{l:[k] \to [n]}
u_{l(1)j(1)}^{\epsilon(1)}\ldots
u_{l(k)j(k)}^{\epsilon(k)}\varphi\big(x_{i(1)l(1)}^{\epsilon(1)}\ldots
x_{i(k)l(k)}^{\epsilon(k)}\big).
\end{eqnarray*}
Note also that any column of a $n \times n$ left $\G$-invariant array of random variables is a $\G$-invariant $n$-tuple.  
\end{rem}

\begin{rem} \label{rem_quantum_subgroups}
If $\G = (A, U)$ and $\mathbb{H} = (B, V)$ are two CMQG's with $U \in M_n(A)$ and $V \in M_n(B)$, we say that $\mathbb{H}$ is a \textit{quantum subgroup} of $\G$ if there exists a surjective $\ast$-homomorphism $\pi:A \to B,$ such that $(id \otimes\pi)(U) = V$.  We note here the elementary fact that a family of random variables $X = \{x_i\}_{i=1}^n \subset (\mathcal A, \varphi)$ with a $\G$-invariant joint $\ast$-distribution also has an $\mathbb{H}$-invariant joint $\ast$-distribution, for any quantum  subgroup $\mathbb H$ of $\G$.  The same is statement is true for arrays $X = \{x_{rs}\}_{1 \le r,s \le n} \subset (\mathcal A, \varphi)$, which have $\G$-bi-invariant joint $\ast$-distributions.  The proof of these facts is an elementary application of the above definitions, and is left to the reader. 
\end{rem}

\subsection{The Hyperoctahedral Quantum Group}

Consider the hypercube $I_n =  [-1,1]^n$ in $\R^n$.  The symmetry group of $I_n$ is called the \textit{hyperoctahedral group}, and is denoted by $H_n$.  We now quantize the definition of $H_n$.
\begin{defn} \label{defn_cubic_unitary}
Let $A$ be a unital $C^\ast$-algebra, and let $U = [u_{rs}]_{1 \le r,s \le n} \in M_n(A)$ be an orthogonal matrix  (i.e. $U$ is unitary, and $u_{rs} = u_{rs}^*$ for all $r,s$.)  We call $U$ a \textit{cubic unitary} if, in addition to orthogonality, we have 
\begin{eqnarray*}
u_{ij}u_{ik} = u_{ji}u_{ki} = 0, && (1 \le i \le n, \ 1 \le j \ne k \le n).
\end{eqnarray*} 
That is, on each row or column of $U$, distinct entries $a \ne b$ satisfy $ab = ba = 0$.
\end{defn}

It is easily shown that $C(H_n)$, the commutative algebra of complex functions on $H_n$, is isomorphic as a C$^*$-algebra to the universal \textit{commutative} C$^*$-algebra $A$, generated by $n^2$ generators $\{u_{rs}\}_{1 \le r,s \le n}$ which satisfy the relations which make $U = [u_{rs}] \in M_n(A)$ a cubic unitary.  By removing the commutativity in the above statement,  we arrive at the definition of the hyperoctahedral quantum group \cite{BaBiCo}.  

\begin{defn} \cite{BaBiCo} \label{defn_hyperoct_quant}
Let $A_H(n)$ denote the universal C$^*$-algebra generated by $n^2$ generators $\{u_{rs}\}_{1 \le r,s \le n}$ subject to the relations which make $U = [u_{rs}] \in M_n(A_H(n))$ a cubic unitary.  The pair $H_n^+ = (A_H(n), U)$ is a CMQG, and  is called the \textit{hyperoctahedral quantum group (of dimension $n$).}  
\end{defn}
By regarding the hypercube $I_n \subset \R^n$ as the  graph formed by $n$ segments ($n$ edges and $2n$ vertices), the quantum group $H_n^+ = (A_H(n), U)$ is indeed the quantum symmetry group of $I_n$ \cite{BaBiCo}. 

There is a useful combinatorial formula describing the Haar state $h:A_H(n) \to \C$, which we now briefly describe:  consider the family of matrices $\{G_{n,2k}\}_{k \in \N}$, where $G_{n,2k}$ is the $|NC_e(2k)| \times |NC_e(2k)|$ matrix indexed by partitions $\pi, \sigma \in NC_e(2k)$, with entries \begin{eqnarray*} G_{n,2k}(\pi,\sigma) = n^{|\pi \vee \sigma|}, && (\pi, \sigma \in NC_e(2k)). \end{eqnarray*} 

\begin{thm} \label{Weingarten_formula} (\cite{BaBiCo, BaCuSp})
For every $k$,  the matrix $G_{n,2k}$ is invertible.  Let $W_{n,2k}$ its matrix inverse.  Then for any pair $i,j:[2k] \to [n]$, we have 
\begin{eqnarray*}
h\big(u_{i(1)j(1)} u_{i(2)j(2)} \ldots u_{i(2k)j(2k)}\big) = \sum_{\substack{\pi, \sigma \in NC_{e}(2k) \\ \ker i \ge \pi, \ \ker j \ge \sigma}}W_{n,2k}(\pi,\sigma),
\end{eqnarray*}
and \begin{eqnarray*}
h\big(u_{i(1)j(1)} u_{i(2)j(2)} \ldots u_{i(2k-1)j(2k-1)}\big) = 0.
\end{eqnarray*}
\end{thm} 

If $X \subseteq (\mathcal A, \varphi)$ is a family (array) with an $H_n^+$(-bi)-invariant joint $\ast$-distribution, Theorem \ref{Weingarten_formula} allows us to get a combinatorial description of which joint $\ast$-moments of the family $X$ must always vanish.

\begin{prop} \label{prop_moments_hyperoct_inv_family} $(i).$  Let $X = \{x_r\}_{r=1}^n$ be a family of random variables in a
NCPS $(\mathcal A, \varphi)$, whose joint
$\ast$-distribution is invariant under the hyperoctahedral quantum group $H_n^+$. Then
for any $k \in \N$, $\epsilon:[k] \to \{1, \ast\}$, and 
$i:[k]\to [n]$, $$\varphi\big(
x_{i(1)}^{\epsilon(1)}x_{i(2)}^{\epsilon(2)}\ldots
x_{i(k)}^{\epsilon(k)}\big) \ne 0$$ \textit{only if} there exists
some $\pi \in NC_e(k)$ such that $\ker i \ge \pi$. \\
$(ii).$  Let $X = \{x_{rs}\}_{1 \le r,s \le n} \subseteq (\mathcal A, \varphi)$ be an array whose joint
$\ast$-distribution is $H_n^+$-bi-invariant. Then 
for any $k \in \N$, $\epsilon:[k] \to \{1, \ast\}$, and any pair
$i, j:[k]\to [n]$, we have   $$\varphi\big(
x_{i(1)j(1)}^{\epsilon(1)}x_{i(2)j(2)}^{\epsilon(2)}\ldots
x_{i(k)j(k)}^{\epsilon(k)}\big) \ne 0$$ \textit{only if} there exist $\pi, \sigma \in NC_e(k)$ such that $\ker i \ge \pi$ and $\ker j \ge \sigma$. \\
$(iii).$  In both $(i)$ and $(ii)$, the family $X$ is identically distributed, and orthogonal in $L^2(\mathcal A, \varphi)$.  
\end{prop}

\begin{proof} We use the formulae given in Remark \ref{rem_on _invariance}.    \\
$(i)$. Since $X = \{x_r\}_{r=1}^n$ is $H_n^+$-invariant, we have $$\varphi\big(x_{i(1)}^{\epsilon(1)}\ldots
x_{i(k)}^{\epsilon(k)}\big)1_{A_H(n)} = \sum_{j:[k] \to [n]}
u_{i(1)j(1)}\ldots
u_{i(k)j(k)}\varphi\big(x_{j(1)}^{\epsilon(1)}\ldots
x_{j(k)}^{\epsilon(k)}\big).$$  Applying the Haar state to both sides of the equation and using Theorem \ref{Weingarten_formula}, we get $$\varphi\big(x_{i(1)}^{\epsilon(1)}\ldots
x_{i(k)}^{\epsilon(k)}\big) = 0,$$ if $k $ is odd, and 
\begin{eqnarray*}
&&\varphi\big(x_{i(1)}^{\epsilon(1)}\ldots
x_{i(k)}^{\epsilon(k)}\big) \\
&=& \sum_{j:[k] \to [n]}
\sum_{\substack{\pi, \sigma \in NC_e(k) \\ \ker i \ge \pi, \ \ker j \ge \sigma}} W_{n,k} (\pi,\sigma) \varphi\big(x_{j(1)}^{\epsilon(1)}\ldots
x_{j(k)}^{\epsilon(k)}\big) \\
&=& \sum_{\substack{ \pi, \sigma \in NC_e(k) \\ \ker i \ge \pi}}W_{n,k} (\pi,\sigma) \sum_{\substack{j:[k] \to [n] \\ \ker j \ge \sigma}} \varphi\big(x_{j(1)}^{\epsilon(1)}\ldots
x_{j(k)}^{\epsilon(k)}\big),
\end{eqnarray*}
if $k$ is even.  Clearly $(i)$ follows from these equalities. \\
$(ii)$.  Since $X = \{x_{rs}\}_{1 \le r,s \le n}$ is $H_n^+$-bi-invariant, we have the two equalities \begin{eqnarray*}
&& \varphi\big(x_{i(1)j(1)}^{\epsilon(1)}\ldots
x_{i(k)j(k)}^{\epsilon(k)}\big)1_{A_H(n)} \\ &=& \sum_{l:[k] \to [n]}
u_{i(1)l(1)}\ldots
u_{i(k)l(k)}\varphi\big(x_{l(1)j(1)}^{\epsilon(1)}\ldots
x_{l(k)j(k)}^{\epsilon(k)}\big) \\ 
&=& \sum_{l:[k] \to [n]}
u_{l(1)j(1)}\ldots
u_{l(k)j(k)}\varphi\big(x_{i(1)l(1)}^{\epsilon(1)}\ldots
x_{i(k)l(k)}^{\epsilon(k)}\big).
\end{eqnarray*}  Applying the Haar state to these equalities and using Theorem \ref{Weingarten_formula} gives, for $k$ odd, $$ \varphi\big(x_{i(1)j(1)}^{\epsilon(1)}\ldots
x_{i(k)j(k)}^{\epsilon(k)}\big) = 0,$$ and \begin{eqnarray*}
&&\varphi\big(x_{i(1)j(1)}^{\epsilon(1)}\ldots
x_{i(k)j(k)}^{\epsilon(k)}\big) \\
&=& \sum_{\substack{ \pi, \sigma \in NC_e(k) \\ \ker i \ge \pi}}W_{n,k} (\pi,\sigma) \sum_{\substack{l:[k] \to [n] \\ \ker l \ge \sigma}} \varphi\big(x_{l(1)j(1)}^{\epsilon(1)}\ldots
x_{l(k)j(k)}^{\epsilon(k)}\big) \\
&=& \sum_{\substack{ \pi, \sigma \in NC_e(k) \\ \ker j \ge \sigma}}W_{n,k} (\pi,\sigma) \sum_{\substack{l:[k] \to [n] \\ \ker l \ge \pi}} \varphi\big(x_{i(1)l(1)}^{\epsilon(1)}\ldots
x_{i(k)l(k)}^{\epsilon(k)}\big),
\end{eqnarray*}
for $k$ even, which clearly implies $(ii)$. \\
$(iii)$.  The fact that the families $X$ in $(i)$ and $(ii)$ are both orthogonal is just a special case, when $k=2$, of $(i)$ and $(ii)$.  Indeed, when $k=2$ we have $G_{n,k} = n$, so the equations in the proof of $(i)$ reduce to $$\varphi(x_{i(1)}x_{i(2)}^*) =  \frac{\delta_{i(1),i(2)}}{n}\sum_{l=1}^n\varphi(x_{l}x_{l}^*),$$ and the equations in the proof of $(ii)$ reduce to 
$$\varphi(x_{i(1)j(1)}x_{i(2)j(2)}^*) =  \frac{\delta_{i(1),i(2)}}{n} \sum_{l=1}^n \varphi(x_{lj(1)}x_{lj(2)}^*) =   \frac{\delta_{j(1),j(2)}}{n} \sum_{l=1}^n \varphi(x_{i(1)l}x_{i(2)l}^*).$$  To show that these families are identically distributed, note that the classical permutation group $S_n$ is a quantum subgroup of $H_n^+$.  Therefore the family $X$ (in either $(i)$ or $(ii)$) has a $S_n$-(bi-)invariant joint $\ast$-distribution, and this implies in particular that the variables in $X$ are identically distributed (see \cite[Section 2]{KoSp} for instance).  
\end{proof}

\begin{rem}
We say that a family $X= \{x_r\}_{r \in \Lambda} \subset (\mathcal A, \varphi)$ ($|\Lambda| \le \infty$), has an $H^+$-invariant joint $\ast$-distribution if every finite subsequence $\{x_{r(l)}\}_{l=1}^n$ of $X$ has an $H_n^+$-invariant joint $\ast$-distribution in the sense of Definition \ref{defn_inv_dist}.
\end{rem}

\section{Strong Haagerup Inequalities} \label{section_main_thm}

In this section we prove Theorem \ref{thm_main_SHI_intro}, which we restate here for convenience.

\begin{thm} \label{main_theorem_tuples}
Let $(\mathcal A, \varphi)$ be a tracial C$^\ast$-probability space, and let $\{x_r\}_{r \in \Lambda} \subset (\mathcal A, \varphi)$ be a family of random variables.  Suppose that the joint $\ast$-distribution of $\{x_r\}_{r \in \Lambda}$ is $H^+$-invariant and invariant under free complexification.  Let $x \in \{x_r\}_{r \in \Lambda}$ be a fixed reference variable.  Then for any homogeneous polynomial \begin{eqnarray*}T = \sum_{i:[d] \to \Lambda} a_i x_{i(1)}x_{i(2)}\ldots x_{i(d)}, && (a_i \in \C),
\end{eqnarray*} of degree $d$ in the variables $\{x_r\}_{r \in \Lambda}$ and any $p \in 2\N \cup \{\infty\}$, we have $$\|T\|_{L^2(\mathcal A, \varphi)} \le \|T\|_{L^p(\mathcal A, \varphi)} \le 4^5 \cdot (3e)^2\sqrt{e}
\frac{\|x\|_{L^p(\mathcal A, \varphi)}^2}{\|x\|_{L^2(\mathcal A, \varphi)}^2} \sqrt{d+1} \|T\|_{L^2(\mathcal A, \varphi)}.$$
\end{thm}

\begin{rem}
As mentioned in the introduction, Theorem \ref{main_theorem_tuples} applies in particular to $\ast$-free, identically distributed R-diagonal families $\{x_r\}_{r \in \Lambda} \subseteq (\mathcal A, \varphi)$.  To see this, note that such a family $\{x_r\}_{r \in \Lambda}$ is $\ast$-free, identically distributed, and each $x_r$ is even (from the definition of R-diagonality).  Therefore it follows from \cite[Proposition 4.3]{BaCuSp} that $\{x_r\}_{r \in \Lambda}$ has an $H^+$-invariant joint $\ast$-distribution.  On the other hand, if $z$ is a Haar unitary which is $\ast$-free from $\{x_r\}_{r \in \Lambda}$, then by Theorem \ref{thm_R_diagonality_free_compl}, $\{zx_r\}_{r \in \Lambda}$ and $\{x_r\}_{r \in \Lambda}$ are identically distributed.  As a consequence, we recover as a special case the strong Haagerup inequality of Kemp-Speicher (\cite[Theorem 1.3]{KeSp}), although with weaker universal constants than theirs.  This is to be expected, as our assumptions on the joint $\ast$-distribution of our operators are weaker.  In Section \ref{quantum_groups}, we will consider some non-trivial examples coming from compact quantum groups, where the greater generality of Theorem \ref{main_theorem_tuples} will be essential. 
\end{rem}

\subsection{Proof of Theorem \ref{main_theorem_tuples}} \label{section_main_theorem_tuples_proof}  Let $\{x_r\}_{r \in \Lambda}$ satisfy the hypotheses of Theorem \ref{main_theorem_tuples}.  Given a function $i:[d] \to \Lambda$, we will use the notation \begin{equation} \label{notation_Xi}
X_i: = x_{i(1)}x_{i(2)}\ldots x_{i(d)}
\end{equation} to denote monomials of degree $d$ in the variables $\{x_r\}_{r \in \Lambda}$.  Fix $d \in \N$, and let  $$T = \sum_{i:[d] \to
\Lambda} a_i X_i,$$ be homogeneous polynomial of degree $d$ as in
Theorem \ref{main_theorem_tuples}. Our strategy is similar to the ones used in \cite{KeSp} and \cite{dS} to prove strong Haagerup inequalities:  we will first express the norms $\|T\|_{2m} := \|T\|_{L^{2m}(\mathcal A,\varphi)}$ ($m \in \N$) in terms of the joint free cumulants of the family $\{x_r\}_{r \in \Lambda}$, then use the assumed properties of the distribution to obtain meaningful estimates.  The inequality for $\|T\|_{\mathcal A} = \|T\|_{L^\infty(\mathcal A, \varphi)}$ will
follow from our estimates by letting $m \to \infty$.  Explicitly we have
\begin{eqnarray*}
&& \|T\|_{2m}^{2m} = \varphi( (TT^*)^{m}) \\
&=& \varphi \Big( \Big(\Big(\sum_{i:[d] \to \Lambda} a_{i} X_i\Big) \Big(\sum_{i:[d] \to \Lambda} a_{i} X_i\Big)^* \Big)^m \Big) \\
&=& \sum_{i_1, \ldots, i_{2m}:[d] \to \Lambda } a_{i_1}\overline{a_{i_2}} \ldots a_{i_{2m-1}}\overline{a_{i_{2m}}}\varphi(X_{i_1}(X_{i_2})^* \ldots X_{i_{2m-1}}(X_{i_{2m}})^*).
\end{eqnarray*}

Now, for any function $i:[d] \to \Lambda$, define the function $\check{i}:[d]
\to \Lambda$ by reversing the order of the $d$-tuple $i$.  That is,
$$(\check{i}(1), \ldots, \check{i}(d)) = (i(d), \dots, i(1)).$$
Then, by a simple change of indices in the above sum, we can write
\begin{eqnarray*}
&&\|T\|_{2m}^{2m} \\
&=& \sum_{i_1, \ldots, i_{2m}:[d] \to \Lambda }
a_{i_1}\overline{a_{\check{i}_2}} \ldots
a_{i_{2m-1}}\overline{a_{\check{i}_{2m}}}\varphi(X_{i_1}(X^*)_{i_2}
\ldots X_{i_{2m-1}}(X^*)_{i_{2m}}),
\end{eqnarray*}
where $$(X^*)_{i_k}:=x_{i_k(1)}^*x_{i_k(2)}^* \ldots x_{i_k(d)}^*.$$
We will now write each moment $\varphi(X_{i_1}(^*)_{i_2} \ldots
X_{i_{2m-1}}(X^*)_{i_{2m}})$ in terms of free cumulants, using equation (\ref{eqn_free_cum_implicit}).  This gives, for
each function $I=(i_1, \ldots, i_{2m}):[2dm] \to \Lambda$,
$$\varphi(X_{i_1}(X^*)_{i_2} \ldots X_{i_{2m-1}}(X^*)_{i_{2m}}) =
\sum_{\pi\in NC(2dm)} \kappa_\pi[I],$$   where \begin{equation*} \label{cum}\kappa_\pi[I] :=
\kappa_\pi[\underbrace{x_{i_1(1)}, \ldots, x_{i_1(d)}}_{\textrm{from } X_{i_1}}, \underbrace{x_{i_2(1)}^*, \ldots, x_{i_2(d)}^*}_{\textrm{from } (X^*)_{i_2}}, \underbrace{ \ldots}_{\ldots} ,  \underbrace{x_{i_{2m}(1)}^*, \ldots, x_{i_{2m}(d)}^*}_{\textrm{from } (X^*)_{i_{2m}}}].
\end{equation*}

This gives the equation \begin{equation} \label{eqn_2mth_moment}
\|T\|_{2m}^{2m} = \sum_{\pi \in NC(2dm)} \sum_{I = (i_1, \ldots, i_{2m}):[2dm] \to \Lambda} a_{i_1}\overline{a_{\check{i}_2}} \ldots a_{i_{2m-1}}\overline{a_{\check{i}_{2m}}} \kappa_\pi[I]. 
\end{equation}

We now use the fact that the joint $\ast$-distribution of $\{x_r\}_{r \in \Lambda}$ is invariant under free complexification, to show that for many $\pi \in NC(2dm)$, the function $I \mapsto \kappa_\pi[I]$ is identically zero.  We state this as the following lemma.  

\begin{lem} \label{lem_star_partitions}  Let $(\mathcal A, \varphi)$ be a NCPS, let $z \in (\mathcal A, \varphi)$ be a Haar unitary which is $\ast$-free from a family $\{x_r\}_{r \in \Lambda}$.  Then for any $k \in \N$, $i:[k] \to \Lambda$, $\epsilon:[k] \to \{1, \ast\},$ and $\pi \in NC(k)$, we have $$\kappa_\pi[(zx_{i(1)})^{\epsilon(1)}, \ldots, (xb_{i(k)})^{\epsilon(k)}] \ne 0$$ only if $\pi \in NC^\epsilon(k).$  In particular, this quantity is zero if $k$ is odd.
\end{lem}

\begin{proof}  From the multiplicative definition of $\kappa_\pi$ ($\pi \in
NC(k)$), it suffices to assume that $\pi = 1_k$.  We then need to
show that $$\kappa_{k}[(zx_{i(1)})^{\epsilon(1)}, \ldots,
(zx_{i(k)})^{\epsilon(k)}] \ne 0$$ only when $\epsilon$ is
alternating and $k$ is even.  When $i(1) = i(2) = \ldots =
i(d)$, this just reduces to the proof of the fact that $zx_{i(1)} \in
(\mathcal A, \varphi)$ is R-diagonal (Theorem \ref{thm_R_diagonality_free_compl} $(i)$).  We refer to the proof of \cite[Proposition 15.8]{NiSp} for an explicit proof of this fact.  For the general case (where $i:[k] \to \Lambda$ is not constant), one just has to notice that the argument in the proof of \cite[Proposition 15.8]{NiSp} still applies when $i$ is non-constant - we just have to carry around the extra indices. 
\end{proof}

Since $\{x_r\}_{r \in \Lambda}$ and $\{zx_r\}_{r \in \Lambda}$ are identically distributed for any Haar unitary $z \in (\mathcal A, \varphi)$ which is $\ast$-free from $\{x_r\}_{r \in \Lambda}$, Lemma \ref{lem_star_partitions} implies that each function $I \mapsto \kappa_\pi[I]$ appearing in equation (\ref{eqn_2mth_moment}) is identically zero for any $\pi \in NC(2dm)\backslash NC^{\epsilon_{d}}(2dm)$, where $\epsilon_d:[2dm] \to \{1,\ast\}$ is the pattern given by
$$\epsilon_d = \overbrace{\underbrace{1,1,\ldots,1}_{d \textrm{
times}}, \underbrace{*,*,\ldots *}_{d \textrm{ times}}, \ldots,
\underbrace{1,1,\ldots,1}_{d \textrm{ times}},
\underbrace{*,*,\ldots *}_{d \textrm{ times}}}^{2m \textrm{
groups}}.$$  (The symbol $\epsilon_d$ will denote the above pattern for the rest of the paper.)  We now get
\begin{equation} \label{eqn_T_2m_norm} \|T\|_{2m}^{2m} = \sum_{\pi \in NC^{\epsilon_d}(2dm)} \sum_{I = (i_1, \ldots, i_{2m}):[2dm] \to \Lambda} a_{i_1}\overline{a_{\check{i}_2}} \ldots a_{i_{2m-1}}\overline{a_{\check{i}_{2m}}} \kappa_\pi[I].
\end{equation} 
When $m=1$, we have $\epsilon_d = (\underbrace{1,1, \ldots, 1}_{d \textrm{ times}},\underbrace{\ast, \ast, \ldots, \ast}_{d \textrm{ times}})$, and it is clear in this case that $NC^{\epsilon_d}(2d)$ contains exactly one partition:  the fully nested pairing $\omega = \{\{1,2d\}, \{2,2d-1\}, \ldots, \{d-1, d+2\}, \{d,d+1\} \}$.  Therefore 
\begin{eqnarray*}
\|T\|_2^2 &=& \sum_{I = (i_1, i_2):[d] \to \Lambda} a_{i_1}\overline{a_{\check{i}_2}} \kappa_\omega[I] \\
&=& \sum_{i_1, i_2:[d] \to \Lambda} a_{i_1}\overline{a_{\check{i}_2}} \prod_{l=1}^d \varphi(x_{i_1(l)}x_{i_2(d - l +1)}^*) \\
&=&  \sum_{i_1, i_2:[d] \to \Lambda}  a_{i_1}\overline{a_{\check{i}_2}} \|x\|_2^{2d}\delta_{i_1, \check{i}_2}  \\
&=& \|x\|_2^{2d} \sum_{i:[d] \to \Lambda} |a_i|^2, 
\end{eqnarray*}
where in the second last line we have used Proposition \ref{prop_moments_hyperoct_inv_family} $(iii)$.  

\begin{rem} \label{rem_orthogonality}
The above calculation tells us that $\{X_i\}_{i:[d] \to \Lambda}$ is an orthogonal system for each $d \in \N$.  Furthermore, Lemma \ref{lem_star_partitions} together with the moment-cumulant formula (\ref{eqn_free_cum_implicit}) tell us that $\varphi(X_iX_{j}^*) = 0$ whenever $i:[d] \to \Lambda, \ j:[d^\prime] \to \Lambda$ are such that $d \ne d^\prime$. I.e. $\{1_\mathcal A\} \cup \bigcup_{d=1}^\infty \{X_i\}_{i:[d] \to \Lambda}$ is an orthogonal system in $L^2(\mathcal A, \varphi)$.   
\end{rem}

We now proceed to the general case $m \ge 2$.  Applying H\"older's inequality to equation (\ref{eqn_T_2m_norm}), we have 
\begin{eqnarray*}
&& \|T\|_{2m}^{2m} \\
&\le& \sum_{\pi \in NC^{\epsilon_d}(2dm)} \sum_{I = (i_1, \ldots, i_{2m}):[2dm] \to \Lambda} |a_{i_1}a_{\check{i}_2} \ldots a_{i_{2m-1}}a_{\check{i}_{2m}}| \cdot |\kappa_\pi[I]| \\
&\le& |NC^{\epsilon_d}(2dm)| \\
&& \times \max_{\pi \in NC^{\epsilon_d}(2dm)}\Big\{  \sum_{I = (i_1, \ldots, i_{2m}):[2dm] \to \Lambda} |a_{i_1}a_{\check{i}_2} \ldots a_{i_{2m-1}}a_{\check{i}_{2m}}| \cdot |\kappa_\pi[I]|\Big\} \\
&\le&  \underbrace{|NC^{\epsilon_d}(2dm)|}_{\textrm{(A)}} \underbrace{\max_{\pi \in NC^{\epsilon_d}(2dm), \ I:[2dm] \to \Lambda}\Big\{ |\kappa_\pi[I]| \Big\}}_{\textrm{(B)}} \\
&& \times \underbrace{\max_{\pi \in NC^{\epsilon_d}(2dm)}\Bigg\{   \sum_{\substack{I:[2dm] \to \Lambda \\ \kappa_\pi[I] \ne 0}} |a_{i_1}a_{\check{i}_2} \ldots a_{i_{2m-1}}a_{\check{i}_{2m}}|\Bigg\}}_{\textrm{(C)}}.
\end{eqnarray*}

So we need to estimate the quantities (A), (B), and (C).  To do
this, we first collect some useful facts about the set
$NC^{\epsilon_d}(2dm).$

\subsection{Properties of $NC^{\epsilon_d}(2dm)$ and Estimates of (A), (B), and (C)} Denote by $NC^{\epsilon_d}_2(2dm) \subseteq NC^{\epsilon_d}(2dm)$ the subset of all pairings in $NC^{\epsilon_d}(2dm)$.  The sets $NC^{\epsilon_d}_2(2dm)$ and $NC^{\epsilon_d}(2dm)$ have been studied extensively in \cite{dS, KeSp, Lar, Or}.  We will only record the results from these articles which are relevant to us.

Denote by $NC(m)^{(d)}$ the set of \textit{$d$-chains} in $NC(m)$:  
$$NC(m)^{(d)}:= \big\{(\sigma_1, \ldots, \sigma_d) \in NC(m)^{d}: \
\sigma_1 \ge \sigma_2 \ge \ldots \ge \sigma_d\big\}.$$  In \cite{Ed},  Edelman studied the combinatorics of the set $NC(m)^{(d)}$, and proved that $|NC(m)^{(d)}| = \frac{1}{m}\left(\begin{array}{c} m(d+1) \\ m-1
\end{array}\right)$.  What is more, it turns out that the sets $NC^{(d)}(m)$ and $NC^{\epsilon_d}_2(2dm)$ are naturally in bijection.   

\begin{lem} \label{lem_fuss_catalan} (\cite[Section 3.1 and Corollary 3.2]{KeSp}, \cite{Or}, \cite{Lar})
For any $d,m \in \N,$ the sets $NC^{\epsilon_d}_2(2dm)$ and $NC(m)^{(d)}$ are in bijection. Therefore $$|NC^{\epsilon_d}_2(2dm)| = |NC(m)^{(d)}| = \frac{1}{m} \left( \begin{array}{c} m(d+1) \\ m-1
\end{array}\right) \le (e(d+1))^m.$$
The last inequality above follows from an application of Sterling's formula.
\end{lem}

To get a handle on the set $NC^{\epsilon_d}(2dm)$, the idea is to compare $NC^{\epsilon_d}(2dm)$ to the subset $NC^{\epsilon_d}_2(2dm)$.  It is a remarkable fact, proved in \cite{dS}, that these two sets are quite ``close'' to being equal.

\begin{prop} \label{prop_fibre_estimate} (\cite[Theorem 1.5]{dS})
For any $d,m \in \N$, we have $$|NC^{\epsilon_d}(2dm)| \le 4^{2m}|NC^{\epsilon_d}_2(2dm)|.$$   Moreover, for any $\pi \in NC^{\epsilon_d}(2dm)$, $\pi$ has at least $dm-2m$ blocks of size two, and the size of any block of $\pi$ is at most $2m$.
\end{prop}  In particular, Proposition \ref{prop_fibre_estimate} and Lemma \ref{lem_fuss_catalan} give the estimate \begin{equation} \label{A_estimate} \textrm{(A)} = |NC^{\epsilon_d}(2dm)| \le 4^{2m}(e(d+1))^m.\end{equation}

We now turn to the estimation of (B).  To do this, we use the following elementary free cumulant bound, variants of which can also be found in \cite[Lemma 3.1]{dS} and \cite[Lemma 4.3]{KeSp}.

\begin{lem} \label{lem_gen_cum_estimate}  Let $\sigma \in NC(n)$, and suppose $\sigma$ has at least $K$ blocks of size $2$ and all blocks $\sigma$ have size at most $N$.  Then for any centered family of random variables $\{b_1,\ldots, b_n\}$ in a tracial C$^\ast$-probability space $(\mathcal A, \varphi)$,  we have $$|\kappa_\sigma[b_1, \ldots, b_n]| \le \Big(\max_{i}\|b_i\|_{L^2(\mathcal A, \varphi)}\Big)^{2K} \Big(16\max_{i}\|b_i\|_{L^N(\mathcal A, \varphi)}\Big)^{n-2K}.$$
\end{lem}

\begin{proof}
Since both sides of the claimed inequality are multiplicative with
respect to the block structure of $\sigma$, it suffices to prove
this result for the case $\sigma = 1_n$.  The case $n=1$ is trivial
since $\kappa_1[b_1] = \varphi[b_1] = 0.$

When $n=2$, we have necessarily $K=1$, and since the $b_i$'s are centered, $\kappa_2[b_1,b_2] = \varphi(b_1b_2) - \varphi(b_1)\varphi(b_2) =  \varphi(b_1b_2)$.  Therefore by the
Cauchy-Schwarz inequality for $\varphi$, we have
\begin{eqnarray*}|\kappa_2[b_1,b_2]| &\le& \|b_1^*\|_{L^2(\mathcal A, \varphi)}\|b_2\|_{L^2(\mathcal A, \varphi)} \\
&=&\|b_1\|_{L^2(\mathcal A, \varphi)}\|b_2\|_{L^2(\mathcal A, \varphi)}  \le \Big(\max_{i}\|b_i\|_{L^2(\mathcal A, \varphi)}\Big)^{2}.
\end{eqnarray*}

When $n \ge 3$, we have $K=0$.  Using equation (\ref{eqn_cums_explicit}), we have $\kappa_n[b_1, \ldots, b_n] = \sum_{\pi \in NC(n)} \mu(\pi,
1_n) \varphi_{\pi}[b_1, \ldots, b_n].$  By taking absolute values and
using the fact that $|\mu(\pi,1_n)| \le C_{n-1}$ for all $\pi \in
NC(n)$,  we get \begin{eqnarray*} |\kappa_n[b_1, \ldots, b_n]| &\le&
C_{n-1}\sum_{\pi \in NC(n)}|\varphi_{\pi}[b_1, \ldots, b_n]| \\
&\le& C_{n}C_{n-1} \max_{\pi \in NC(n)}|\varphi_{\pi}[b_1, \ldots,
b_n]| \\
&\le& C_n C_{n-1} \Big(\max_{i}\|b_i\|_{L^N(\mathcal A,
\varphi)}\Big)^{n},
\end{eqnarray*}
where in the last line we have applied H\"olders inequality (for the trace $\varphi$) to the
quantities $|\varphi_\pi[b_1, \ldots, b_n]|$, and used the fact
that $\|x\|_{L^r(\mathcal A,\varphi)} \le \|x\|_{L^{N}(\mathcal A,
\varphi)}$, for all $x \in \mathcal A$ and $r \le N$.  The proof is
now completed by noting that $C_k \le 4^k$ for any $k$.
\end{proof}

Using Lemma \ref{lem_gen_cum_estimate}, we obtain the
following estimate for (B).

\begin{cor} \label{cor_cum_estimate}
For any $\pi \in NC^{\epsilon_d}(2dm)$ and any function $I:[2dm] \to
\Lambda$, \begin{equation} \label{B_estimate} |\kappa_\pi[I]| \le \|x\|_{2}^{2dm}\Big(
\frac{16\|x\|_{2m}}{\|x\|_2} \Big)^{4m}.
\end{equation}
\end{cor}

\begin{proof}
From Lemma \ref{lem_gen_cum_estimate} and the fact that the variables $\{x_r\}_{r\in \Lambda}$ are identically distributed according to our reference variable $x$, we have
$$|\kappa_\pi[I]| \le \|x\|_2^{2K}(16\|x\|_{2m})^{2dm - 2K},$$ where $K$
denotes the number of blocks of $\pi$ which are pairings.  Now,
since $\pi \in NC^{\epsilon_d}(2dm)$, Proposition
\ref{prop_fibre_estimate} implies that $K \ge dm - 2m$. Using this
estimate in the above inequality we have
\begin{eqnarray*}
\|x\|_2^{2K}(16\|x\|_{2m})^{2dm - 2K} &=& \|x\|_2^{2dm}\Big(\frac{16\|x\|_{2m}}{\|x\|_{2}}\Big)^{2dm
- 2K} \\
&\le&
\|x\|_2^{2dm}\Big(\frac{16\|x\|_{2m}}{\|x\|_{2}}\Big)^{4m}.
\end{eqnarray*}
\end{proof}

We will now prove the inequality \begin{equation} \label{C_estimate} \textrm{(C)} = \max_{\pi \in
NC^{\epsilon_d}(2dm)}\Bigg\{   \sum_{\substack{I:[2dm] \to \Lambda \\
\kappa_\pi[I] \ne 0}} |a_{i_1}a_{\check{i}_2} \ldots
a_{i_{2m-1}}a_{\check{i}_{2m}}|\Bigg\} \le (3e)^{4m}  \Big(\sum_{i:[d]
\to \Lambda} |a_i|^2\Big)^m.
\end{equation} To do this, we will need the following two lemmas.

\begin{lem} \label{lem_est_partitions_smaller_pi}
Let $\pi \in NC^{\epsilon_d}(2dm)$, and fix $\sigma \le \pi$.  Then $\sigma \in NC^{\epsilon_d}(2dm)$ if and only if $\sigma \in NC_e(2dm)$.  Furthermore, $$|\{\sigma \in NC^{\epsilon_d}(2dm): \ \sigma \le \pi\}| \le |NC(2m)^{(2)}| \le (3e)^{2m}.$$
\end{lem}

\begin{proof}  Fix $\pi \in NC^{\epsilon_d}(2dm)$, and let $\sigma \le \pi$.  The ``only if'' part of the first statement follows from the definition of $NC^{\epsilon_d}(2dm)$.  To prove the ``if'' part, suppose $\sigma \in NC_e(2dm)$.  Let $V = \{j_1 < j_2 < \ldots < j_{2r}\}$ be a block of $\sigma$.  Since $\sigma$ is even and non-crossing, a simple inductive argument shows that the parity of the sequence $j_1,j_2, \ldots, j_{2r}$ must be alternating.  Now let $W$ be the block of $\pi$ containing $V$.  Then by definition $\epsilon_d|_W$ is alternating, and since the parity of the elements of $V \subseteq W$ alternates as well, $\epsilon_d|_V = (\epsilon_d(j_1), \ldots,  \epsilon_d(j_{2r}))$ is also alternating.  As $V$ was arbitrary, we get $\sigma \in NC^{\epsilon_d}(2dm)$.  

We now prove the second statement of the lemma.  By Proposition \ref{prop_fibre_estimate}, $\pi$ contains at least $dm-2m$ blocks of size $2$.  Let $A \subseteq [2dm]$ denote the union of these blocks of size $2$.  For any $\sigma \in NC^{\epsilon_d}(2dm) \subseteq NC_e(2dm)$ with $\sigma \le \pi$, we must then have $\sigma|_A = \pi|_A$, and $\sigma|_{[2dm]\backslash A} \le \pi|_{[2dm]\backslash A} \in NC_e([2dm]\backslash A)$.  Therefore we clearly have an injection
\begin{eqnarray*} \{\sigma \in NC^{\epsilon_d}(2dm): \ \sigma \le
\pi\} &\hookrightarrow&
NC_e([2dm]\backslash A), \\
\sigma &\mapsto& \sigma|_{[2dm]\backslash A}.
\end{eqnarray*}
This gives $$|\{\sigma \in NC^{\epsilon_d}(2dm): \ \sigma \le \pi\}|
\le |NC_e([2dm]\backslash A)| \le |NC_e(4m)|,$$ since $|[2dm] \backslash A| \le 4m$.  To complete the proof, we use the fact that there is a bijection between $NC_e(4m)$ and the set of $2$-chains $NC(2m)^{(2)}$ (cf. \cite[Exercise 9.42]{NiSp}), together with the bound on $|NC(2m)^{(2)}|$ given by Lemma \ref{lem_fuss_catalan}. 
\end{proof}

\begin{lem} \label{lem_scalars_partial_sum_estimate}
Let $\{a_{i}\}_{i:[d] \to \Lambda} \subset \C$ be a finitely supported family of complex
numbers, and let $\pi \in \mathcal P^{\epsilon_d}(2dm)$. Then
$$\sum_{\substack{I:[2dm] \to \Lambda \\ \ker I \ge \pi}}
|a_{i_1}a_{i_2} \ldots a_{i_{2m-1}}a_{i_{2m}}| \le
\Big(\sum_{i:[d] \to \Lambda} |a_i|^2\Big)^m.$$
\end{lem}

\begin{proof}
Given $\pi
\in \mathcal P^{\epsilon_d}(2dm)$, associate to it a pairing $\pi_r
\in \mathcal P^{\epsilon_d}_2(2dm)$, as follows: for each block $V = \{k_1 <
k_2 < \ldots < k_{2s}\}$ of $\pi$, declare $\{k_1, k_2\}, \ldots,
\{k_{2r-1}, k_{2s}\}$ to be blocks of $\pi_r$. Observe that $\pi_r
\le \pi$ by construction.  In particular, for all $I:[2dm] \to \Lambda$,
$$\ker I \ge \pi \implies \ker I \ge \pi_r.$$  Therefore
\begin{eqnarray*} \sum_{\substack{I:[2dm] \to \Lambda \\ \ker I \ge \pi}}
|a_{i_1}a_{i_2} \ldots a_{i_{2m-1}}a_{i_{2m}}| &\le&
\sum_{\substack{I:[2dm] \to \Lambda \\ \ker I \ge \pi_r}}
|a_{i_1}a_{i_2} \ldots a_{i_{2m-1}}a_{i_{2m}}| \\
&=& \sum_{\substack{I:[2dm] \to \Lambda \\ \ker I \ge \pi_r}}
\prod_{k=1}^m|a_{i_{2k-1}}| \prod_{k=1}^m |a_{i_{2k}}|.
\end{eqnarray*}

Let $B_1: = \epsilon_d^{-1}\{1\}$ and $B_\ast: = \epsilon_d^{-1}\{\ast\}$, so that $[2dm] = B_1 \dot \cup B_\ast$.  Since $\pi_r \in \mathcal P^{\epsilon_d}(2dm)$, we can identify $\pi_r$ with the bijection \begin{eqnarray*}\pi_r:B_1 \to B_\ast, && \pi_r(t) = s \iff \{t,s\} \textrm{ is a block of } \pi_r.
\end{eqnarray*}  
Now fix a function $I = (i_1, \ldots, i_{2m}):[2dm] \to \Lambda$.  Let $I_1 = (i_1,i_3, \ldots, i_{2m-3}, i_{2m-1}) = I|_{B_1}$, and let $I_\ast = (i_2,i_4, \ldots, i_{2m-2}, i_{2m}) = I|_{B_\ast}$.  Then we can rewrite the condition $\ker I \ge \pi_r$ as $I_\ast = I_1 \circ \pi_r.$  Writing  $A_{I_1} = \prod_{k=1}^m|a_{i_{2k-1}}|$ and $A_{I_\ast} = \prod_{k=1}^m |a_{i_{2k}}|$ our previous inequality now becomes
\begin{eqnarray*}
&&\sum_{\substack{I:[2dm] \to \Lambda \\ \ker I \ge \pi}}
|a_{i_1}a_{i_2} \ldots a_{i_{2m-1}}a_{i_{2m}}|  \le \sum_{\substack{I:[2dm] \to \Lambda \\
\ker I \ge \pi_r}} A_{I_1}A_{I_\ast} \\
&=& \sum_{I_1 = (i_1, i_3,
\dots, i_{2m-1})} A_{I_1} A_{I_1 \circ \pi_r}
\\
&\le& \Big(\sum_{I_1 = (i_1, i_3, \dots, i_{2m-1})}
A_{I_1}^2\Big)^{1/2} \Big(\sum_{I_1 = (i_1, i_3, \dots,
i_{2m-1})}A_{I_1 \circ \pi_r}^2\Big)^{1/2} \\
&& (\textrm{by the Cauchy-Schwarz inequality}) \\
&=& \sum_{I_1 = (i_1, i_3, \dots, i_{2m-1})} A_{I_1}^2 \ \ \ (\textrm{since $\pi_r$ is a bijection}) \\
&=& \sum_{i_1, i_3, \ldots, i_{2m-1}:[d] \to \Lambda}
|a_{i_1}|^2|a_{i_{3}}|^2\ldots |a_{i_{2m-1}}|^2 = \Big(\sum_{i:[d] \to \Lambda} |a_i|^2\Big)^m.\\
\end{eqnarray*}
\end{proof}

Using Lemmas \ref{lem_est_partitions_smaller_pi} and \ref{lem_scalars_partial_sum_estimate}, we are now ready to obtain the bound (\ref{C_estimate}).  Fix $\pi \in NC^{\epsilon_d}(2dm)$ and consider the sum $$\sum_{\substack{I:[2dm] \to \Lambda \\ \kappa_\pi[I] \ne 0}} |a_{i_1}a_{\check{i}_2} \ldots a_{i_{2m-1}}a_{\check{i}_{2m}}|.$$  For any $I (i_1, i_2, \ldots , i_{2m}):[2dm] \to \Lambda,$ we can use equation (\ref{eqn_free_cum_implicit_full}) to write $$\kappa_\pi[I] = \sum_{\substack{\sigma \in NC(2dm) \\ \sigma \le \pi}} \mu(\sigma, \pi) \varphi_\sigma[I],$$ 
\begin{equation} \label{eqn_varphi_sigma}
\varphi_\sigma[I] := \varphi_\sigma[x_{i_1(1)}, \ldots, x_{i_1(d)}, x_{i_2(1)}^*, \ldots, x_{i_2(d)}^*, \ldots ,  x_{i_{2m}(1)}^*, \ldots, x_{i_{2m}(d)}^*].
\end{equation}
In particular, $\kappa_\pi[I] \ne 0$ \textit{only if} there exists some $\sigma \in  NC(2dm)$ with $\sigma \le \pi$ such that $\varphi_\sigma[I] \ne 0$.  Fix such a $\sigma \le \pi$.  Since $\{x_r\}_{r \in \Lambda}$ has an even joint $\ast$-distribution  (Proposition \ref{prop_moments_hyperoct_inv_family}), it follows that $\varphi_\sigma[I] \ne 0$ \textit{only if} each block of $\sigma$ has even cardinality.  I.e.  $\sigma \in NC_e(2dm)$, which is equivalent to $\sigma \in NC^{\epsilon_d}(2dm)$ by Lemma \ref{lem_est_partitions_smaller_pi}.  Therefore, by simply over-counting, we have the estimate \begin{eqnarray*} && \sum_{\substack{I:[2dm] \to \Lambda \\ \kappa_\pi[I] \ne 0}} |a_{i_1}a_{\check{i}_2} \ldots a_{i_{2m-1}}a_{\check{i}_{2m}}| \\
&\le& \sum_{\substack{\sigma  \in NC^{\epsilon_d}(2dm) \\ \sigma \le \pi}}\sum_{\substack{I:[2dm] \to \Lambda \\ \varphi_{\sigma}[I] \ne 0}}|a_{i_1}a_{\check{i}_2} \ldots a_{i_{2m-1}}a_{\check{i}_{2m}}| \\
&\le& |\{\sigma \in NC^{\epsilon_d}(2dm): \ \sigma \le \pi \}|  \\
&& \times \max_{\substack{ \sigma \in NC^{\epsilon_d}(2dm): \\  \sigma \le \pi}} \Bigg\{ \sum_{\substack{I:[2dm] \to \Lambda \\  \varphi_{\sigma}[I] \ne 0}}|a_{i_1}a_{\check{i}_2} \ldots a_{i_{2m-1}}a_{\check{i}_{2m}}| \Bigg\} \\
&\le& (3e)^{2m} \max_{\substack{ \sigma \in NC^{\epsilon_d}(2dm): \\  \sigma \le \pi}} \Bigg\{ \sum_{\substack{I:[2dm] \to \Lambda \\  \varphi_{\sigma}[I] \ne 0}}|a_{i_1}a_{\check{i}_2} \ldots a_{i_{2m-1}}a_{\check{i}_{2m}}| \Bigg\},
\end{eqnarray*}
where in the last line we have used Lemma \ref{lem_est_partitions_smaller_pi}.

Now fix $\sigma \in NC^{\epsilon_d}(2dm)$ with $\sigma \le \pi$.  By applying Proposition \ref{prop_moments_hyperoct_inv_family} $(i)$ to each block $V$ of $\sigma$, we get $\varphi_\sigma[I] \ne 0$ \textrm{only if} there exists some
$\rho \in NC_e(2dm)$ such that $\ker I \ge \rho$ and $\rho \le
\sigma$.  (In particular,  $\rho \in NC^{\epsilon_d}(2dm)$, by Lemma
\ref{lem_est_partitions_smaller_pi}).  Therefore, by possibly over-counting again, we
have the uniform estimate
\begin{eqnarray*}
&&\sum_{\substack{I:[2dm] \to \Lambda \\  \varphi_{\sigma}[I] \ne
0}}|a_{i_1}a_{\check{i}_2} \ldots a_{i_{2m-1}}a_{\check{i}_{2m}}| \\
&\le& \sum_{\substack{\rho \in NC^{\epsilon_d}(2dm): \\ \rho \le
\sigma}} \sum_{\substack{I:[2dm] \to \Lambda \\  \ker I \ge
\rho}}|a_{i_1}a_{\check{i}_2} \ldots a_{i_{2m-1}}a_{\check{i}_{2m}}|
\\ &\le& |\{\rho \in NC^{\epsilon_d}(2dm): \ \rho \le \sigma\}|  \\
&& \times \max_{\substack{ \rho \in NC^{\epsilon_d}(2dm): \\
\rho \le \sigma}} \Bigg\{ \sum_{\substack{I:[2dm] \to \Lambda \\
\ker I \ge \rho}}|a_{i_1}a_{\check{i}_2} \ldots
a_{i_{2m-1}}a_{\check{i}_{2m}}| \Bigg\} \\
&\le& (3e)^{2m} \Big(\sum_{i:[d] \to \Lambda} |a_i|^2\Big)^m
\end{eqnarray*}
where in the last line we have used both Lemmas
\ref{lem_est_partitions_smaller_pi} and
\ref{lem_scalars_partial_sum_estimate}.  This completes the proof of inequality (\ref{C_estimate}).  Note that the presence of the ``checks'' $\check{i}_{2k}$, $1 \le k \le m$, in the previous sums do not effect the applicability of Lemma \ref{lem_scalars_partial_sum_estimate}, since for any $\rho \in NC^{\epsilon_d}(2dm)$ we can write $$\sum_{\substack{I:[2dm] \to \Lambda \\
\ker I \ge \rho}}|a_{i_1}a_{\check{i}_2} \ldots
a_{i_{2m-1}}a_{\check{i}_{2m}}| = \sum_{\substack{I:[2dm] \to \Lambda \\
\ker I \ge \rho^\prime}}|a_{i_1}a_{i_2} \ldots
a_{i_{2m-1}}a_{i_{2m}}|,$$ where $\rho^\prime \in \mathcal P^{\epsilon_d}(2dm)$ is the unique partition with the property that \begin{eqnarray*} \ker (i_1, \ldots, i_{2m}) \ge \rho \Longleftrightarrow \ker (i_1, \check{i_2}, \ldots, i_{2m-1}, \check{i}_{2m}) \ge \rho^\prime,  \end{eqnarray*} for all $I = (i_1, \ldots, i_{2m}):[2dm] \to \Lambda$.

Putting inequalities (\ref{A_estimate}), (\ref{B_estimate}), and (\ref{C_estimate}) together we get
\begin{eqnarray*}
\|T\|_{2m}^{2m} &\le& \textrm{(A)} \cdot \textrm{(B)} \cdot \textrm{(C)}  \\
&\le& 4^{2m}(e(d+1))^m \cdot
\|x\|_{2}^{2dm}\Big( \frac{16\|x\|_{2m}}{\|x\|_2}\Big)^{4m} \cdot (3e)^{4m} \Big(\sum_{i:[d] \to \Lambda} |a_i|^2\Big)^m\\
&=&4^{10m}3^{4m} e^{5m} (d+1)^m\Big( \frac{\|x\|_{2m}}{\|x\|_2} \Big)^{4m} \|T\|_2^{2m}.
\end{eqnarray*}
After taking $2m$'th roots we get $$\|T\|_{2m} \le 4^5 \cdot (3e)^2\sqrt{e}  \frac{\|x\|_{2m}^2}{\|x\|_2^2} \sqrt{(d+1)} \|T\|_2, \ \ \ (m \in \N).$$ This completes the proof of Theorem \ref{main_theorem_tuples}, since the lower bound $\|T\|_2 \le \|T\|_{2m}$ is automatic.

\section{Strong Haagerup Inequalities for Bi-Invariant Arrays.} \label{section_bi_invariant_arrays}

In this section, we extend Theorem \ref{main_theorem_tuples} to the
case where we have an $n \times n$ array $\{x_{rs}\}_{1 \le r,s \le
n}$ of random variables in a tracial $C^\ast$-probability space $(\mathcal A, \varphi)$, whose joint
$\ast$-distribution is $H_n^+$-bi-invariant.  We will need this extension to prove our strong Haagerup inequalities for $U_n^+$ in Section \ref{quantum_groups}.  Here is the main result.

\begin{thm} \label{main_theorem_arrays}
Let $(\mathcal A, \varphi)$ be a tracial C$^\ast$-probability space, and suppose $\{x_{rs}\}_{1 \le r,s \le n} \subset (\mathcal A, \varphi)$ is an $n \times n$ array whose joint $\ast$-distribution is $H^+_n$-bi-invariant and invariant under free complexification.  Let $x \in \{x_{r,s}\}_{1 \le r,s \le n}$ be a fixed reference variable.  Then for any homogeneous polynomial \begin{eqnarray*}T = \sum_{k,l:[d] \to [n]} a_{k,l} x_{k(1)l(1)}x_{k(2)l(2)}\ldots x_{k(d)l(d)}, && (a_{k,l} \in \C),
\end{eqnarray*} of degree $d$ in the variables $\{x_{rs}\}_{1 \le r,s \le n}$ and any $p \in 2\N \cup \{\infty\}$, 
we have $$\|T\|_{L^2(\mathcal A, \varphi)} \le \|T\|_{L^p(\mathcal A, \varphi)} \le 4^5 \cdot (3e)^3\sqrt{e}
\frac{\|x\|_{L^p(\mathcal A, \varphi)}^2}{\|x\|_{L^2(\mathcal A,
\varphi)}^2} \sqrt{d+1} \|T\|_{L^2(\mathcal A, \varphi)}.$$ 
\end{thm}

\begin{rem} \label{rem_additional factor note}
Note that there is an additional factor of $3e$ in the constant of Theorem
\ref{main_theorem_arrays} that does not appear in Theorem
\ref{main_theorem_tuples}. 
\end{rem}

\begin{proof} The proof of this result is almost identical to that of Theorem \ref{main_theorem_tuples}, so we only sketch it.   Fix $d \in \N$, and $T$ as in the statement of the theorem.  Let $\Lambda = [n] \times [n]$, and identify any pair of functions $k,l:[d] \to [n]$ with the function $i:[d] \to \Lambda$ given by \begin{eqnarray*} i(j) = (k(j),l(j)), && (j \in [d] ).
\end{eqnarray*} 
With this change of indices, $T$ can be written as \begin{eqnarray*}T =\sum_{i:[d] \to \Lambda} a_i x_{i(1)}x_{i(2)}\ldots x_{i(d)}, && (a_i \in \C).
\end{eqnarray*}
Repeating the arguments in the proof of Theorem \ref{main_theorem_tuples} (and using the same notation therein), we get the same (in)equalities we had there: $$\|T\|_2^2 =  \|x\|_{2}^{2d} \sum_{i:[d] \to \Lambda}|a_i|^2 =  \|x\|_{2}^{2d} \sum_{k,l:[d] \to [n]}|a_{k,l}|^2,$$ and for general $m \in \N$,
\begin{eqnarray*}
\|T\|_{2m}^{2m} &\le&  \underbrace{|NC^{\epsilon_d}(2dm)|}_{\textrm{(A$^\prime$)}} \underbrace{\max_{\pi \in NC^{\epsilon_d}(2dm), \ I:[2dm] \to \Lambda}\Big\{ |\kappa_\pi[I]| \Big\}}_{\textrm{(B$^\prime$)}} \\
&& \times \underbrace{\max_{\pi \in NC^{\epsilon_d}(2dm)}\Bigg\{   \sum_{\substack{I:[2dm] \to \Lambda \\ \kappa_\pi[I] \ne 0}} |a_{i_1}a_{\check{i}_2} \ldots a_{i_{2m-1}}a_{\check{i}_{2m}}|\Bigg\}}_{\textrm{(C$^{\prime}$)}}.
\end{eqnarray*}  
These (in)equalities remain valid in our new setting because, up to this point in the proof, they only rely on the fact that our random variables are orthogonal in $L^2(\mathcal A, \varphi)$, identically distributed, and that their joint $\ast$-distribution is invariant under free complexification.

Consider the quantities (A), (B), and (C) defined in Section \ref{section_main_theorem_tuples_proof}.  Obviously we have (A$^\prime$) $=$ (A).  Also note that the bound (\ref{B_estimate}) obtained for (B) only relied on the structure of collection of partitions $NC^{\epsilon_d}(2dm)$, so the same bound applies to (B$^\prime$):   \begin{equation}\label{Bprime_estimate}
\textrm{(B$^\prime$)} \le \|x\|_{2}^{2dm}\Big(
\frac{16\|x\|_{2m}}{\|x\|_2} \Big)^{4m}.
\end{equation}
For (C$^\prime$) we will prove the inequality
\begin{equation} \label{Cprime_estimate}
\textrm{(C$^\prime$)} \le (3e)^{6m}  \Big(\sum_{k,l:[d]
\to [n]} |a_{k,l}|^2\Big)^m.
\end{equation} 
Theorem \ref{main_theorem_arrays} now follows from the inequalities (\ref{A_estimate}), (\ref{Bprime_estimate}), and (\ref{Cprime_estimate}).  The proof of (\ref{Cprime_estimate}) is only a minor modification of the proof of inequality (\ref{C_estimate}) for (C).  Indeed, for any $\pi \in NC^{\epsilon_d}(2dm),$ equation (\ref{eqn_free_cum_implicit_full}) and Lemma \ref{lem_est_partitions_smaller_pi} give (just as before) $$\kappa_\pi[I] = \sum_{\substack{\sigma \in NC^{\epsilon_d}(2dm) \\ \sigma \le \pi}} \mu(\sigma,\pi) \varphi_\sigma[I],$$ and therefore
 \begin{eqnarray*} && \sum_{\substack{I:[2dm] \to \Lambda \\ \kappa_\pi[I] \ne 0}} |a_{i_1}a_{\check{i}_2} \ldots a_{i_{2m-1}}a_{\check{i}_{2m}}| \\
&\le& \sum_{\substack{\sigma  \in NC^{\epsilon_d}(2dm) \\ \sigma \le \pi}}\sum_{\substack{I:[2dm] \to \Lambda \\ \varphi_{\sigma}[I] \ne 0}}|a_{i_1}a_{\check{i}_2} \ldots a_{i_{2m-1}}a_{\check{i}_{2m}}| \\
&\le& (3e)^{2m} \max_{\substack{ \sigma \in NC^{\epsilon_d}(2dm): \\  \sigma \le \pi}} \Bigg\{ \sum_{\substack{I:[2dm] \to \Lambda \\  \varphi_{\sigma}[I] \ne 0}}|a_{i_1}a_{\check{i}_2} \ldots a_{i_{2m-1}}a_{\check{i}_{2m}}| \Bigg\},
\end{eqnarray*}
where in the last line we have applied Lemma \ref{lem_est_partitions_smaller_pi}.  Fix $\sigma \in NC^{\epsilon_d}(2dm)$ with $\sigma \le \pi$, and write
\begin{eqnarray*}
&& \sum_{\substack{I:[2dm] \to \Lambda \\  \varphi_{\sigma}[I] \ne 0}}|a_{i_1}a_{\check{i}_2} \ldots a_{i_{2m-1}}a_{\check{i}_{2m}}| \\
&=&\sum_{\substack{K,L:[2dm] \to [n] \\  \varphi_{\sigma}[K,L] \ne 0}}|a_{k_1,l_1}a_{\check{k}_2, \check{l}_2} \ldots a_{k_{2m-1},
l_{2m-1}}a_{\check{k}_{2m}, \check{l}_{2m}}|,
\end{eqnarray*}
where  \begin{eqnarray*}
\varphi_\sigma[K,L] &=& \varphi_\sigma[x_{k_1(1)l_1(1)}, \ldots, x_{k_1(d)l_1(d)}, \ldots, x_{k_{2m}(1)l_{2m}(1)}^*, \ldots, x_{k_{2m}(d)l_{2m}(d)}^*] \\
&=& \varphi_\sigma[I].\end{eqnarray*}
Next, observe that $H_n^+$-bi-invariance (i.e. Proposition \ref{prop_moments_hyperoct_inv_family} $(ii)$ applied to each block of $\sigma$) implies that $\varphi_\sigma[K,L] \ne 0$ only if there exist $\rho, \delta \in NC_e(2dm)$ such that $\rho, \delta \le \sigma$ (so $\rho, \delta \in NC^{\epsilon_d}(2dm)$ by Lemma \ref{lem_est_partitions_smaller_pi}) with the property that $\ker K \ge \rho$, and $\ker L \ge \sigma$.  Using this fact, we can over count the possible non-zero contributions to the above sum and get 
\begin{eqnarray*}
&& \sum_{\substack{K,L:[2dm] \to [n] \\  \varphi_{\sigma}[K,L] \ne 0}}|a_{k_1,l_1}a_{\check{k}_2, \check{l}_2} \ldots a_{k_{2m-1},
l_{2m-1}}a_{\check{k}_{2m}, \check{l}_{2m}}| \\
&\le& \sum_{\substack{\rho, \delta \in NC^{\epsilon_d}(2dm) \\ \rho, \delta \le \sigma}} \sum_{\substack{K,L:[2dm] \to [n] \\\ker I \ge \rho, \ \ker J \ge \delta}}|a_{k_1,l_1}a_{\check{k}_2, \check{l}_2} \ldots a_{k_{2m-1},
l_{2m-1}}a_{\check{k}_{2m}, \check{l}_{2m}}|  \\
&\le& |\{\rho \in NC^{\epsilon_d}(2dm): \ \rho \le \sigma \}|^2 \\
&& \times \max_{\rho, \delta \in NC^{\epsilon_d}(2dm)} \Bigg\{ \sum_{\substack{K,L:[2dm] \to [n] \\ \ker K \ge \rho, \ \ker L \ge \delta}}|a_{k_1,l_1}a_{\check{k}_2, \check{l}_2} \ldots a_{k_{2m-1},
l_{2m-1}}a_{\check{k}_{2m}, \check{l}_{2m}}|  \Bigg\} \\
&\le&(3e)^{4m} \\
&& \times \max_{\rho, \delta \in NC^{\epsilon_d}(2dm)} \Bigg\{ \sum_{\substack{K,L:[2dm] \to [n] \\ \ker K \ge \rho, \ \ker L \ge \delta}}|a_{k_1,l_1}a_{\check{k}_2, \check{l}_2} \ldots a_{k_{2m-1},
l_{2m-1}}a_{\check{k}_{2m}, \check{l}_{2m}}|  \Bigg\}, 
\end{eqnarray*}
where in the last inequality we have used Lemma \ref{lem_est_partitions_smaller_pi}.  Finally, we rewrite the sums \begin{eqnarray*}
\sum_{\substack{K,L:[2dm] \to [n] \\\ker K \ge \rho, \ \ker L \ge \delta}}|a_{k_1,l_1}a_{\check{k}_2, \check{l}_2} \ldots a_{k_{2m-1},
l_{2m-1}}a_{\check{k}_{2m}, \check{l}_{2m}}|, && ( \rho, \delta \in NC^{\epsilon_d}(2dm)),
\end{eqnarray*} appearing above, in the equivalent form \begin{eqnarray*}
\sum_{\substack{J = (k_1, l_1, k_2, l_2, \ldots , k_{2m},l_{2m}):[2(2d)m] \to [n] \\ \ker J \ge \rho \sqcup \delta}}|a_{k_1,l_1}a_{\check{k}_2, \check{l}_2} \ldots a_{k_{2m-1},
l_{2m-1}}a_{\check{k}_{2m}, \check{l}_{2m}}|,
\end{eqnarray*} where $\rho \sqcup \delta \in \mathcal P([2dm] \sqcup [2dm]) \cong \mathcal P (4dm)$ is the \textit{disjoint union} of $\rho$ and $\delta$.  Since it is obvious that $\rho \sqcup \delta \in \mathcal P^{\epsilon_{2d}}(4dm)$, we may apply Lemma \ref{lem_scalars_partial_sum_estimate} (the same way we did in the proof of inequality (\ref{C_estimate})) to the above sums obtain
\begin{eqnarray*}  \sum_{\substack{K,L:[2dm] \to [n] \\\ker K \ge \rho, \ \ker L \ge \delta}}|a_{k_1,l_1}a_{\check{k}_2, \check{l}_2} \ldots a_{k_{2m-1},
l_{2m-1}}a_{\check{k}_{2m}, \check{l}_{2m}}|
&\le& \Big(\sum_{k,l:[d] \to [n]} |a_{k,l}|^2\Big)^m,
\end{eqnarray*} for all $\rho, \delta \in NC^{\epsilon_d}(2dm)$. Putting all these uniform estimates together gives the bound (\ref{Cprime_estimate}).
\end{proof}

\section{Application to the Metric Approximation Property} \label{section_MAP}

In this section we study the norm-closed, non-self-adjoint (unital) operator algebra $\mathcal B$ generated by a family of random variables satisfying the hypotheses of Theorems \ref{main_theorem_tuples} or \ref{main_theorem_arrays}.  We show that $\mathcal B$ always has the metric approximation property, and derive some other intermediate results along the way, which may be of independent interest.  We begin by recalling the definition of the metric approximation property. 

\begin{defn}\label{def_MAP}
Let $Y$ be a Banach space.  We say that $Y$ has the \textit{metric approximation property (MAP)} if there exists a net $\{S_\alpha\}_{\alpha \in A} \subset B(Y)$ of finite rank contractions converging to the identity map in the strong operator topology (s.o.t.) on $B(Y)$.  That is, for all $y \in Y$ $$\lim_{\alpha \in A}\|S_\alpha y - y\| = 0.$$ \end{defn}  

Let $(\mathcal A, \varphi)$ be a tracial C$^\ast$-probability space, and let $X = \{x_r\}_{r \in \Lambda}$ be a family of operators satisfying the hypotheses of Theorem \ref{main_theorem_tuples} or Theorem \ref{main_theorem_arrays}.  Denote by $\mathcal B = \mathcal B_X \subseteq \mathcal A$ the norm-closed, non-self-adjoint, unital operator algebra generated by $\{x_r\}_{r \in \Lambda}$.  We prove:

\begin{thm} \label{thm_MAP}
The operator algebra $\mathcal B$ has the metric approximation property.
\end{thm}

To prove Theorem \ref{thm_MAP}, we use a fairly standard truncation argument, originating from Haagerup in \cite{Ha}.  We start with some notation and two preliminary results.

Let $(\mathcal A, \varphi)$ be a C$^\ast$-probability space (not necessarily tracial), let $X = \{x_r\}_{ r \in \Lambda} \subset (\mathcal A, \varphi)$ be a family of random variables, and assume that the joint $\ast$-distribution of $X$ is invariant under free complexification. Let $\mathcal B = \mathcal B_X \subseteq \mathcal A$ be the norm-closed, non-self-adjoint, unital operator algebra generated by $X$.  Denote by  $$L^2(\mathcal B) := \overline{\mathcal B}^{\|\cdot\|_2} \subseteq L^2(\mathcal A, \varphi),$$ the Hilbert space generated by $\mathcal B$, and for $d \in \N \cup \{0\}$, let $P_d:L^2(\mathcal B) \to L^2_d(\mathcal B)$ be the orthogonal projection onto the degree $d$ subspace $$L^2_d(\mathcal B):= \overline{\textrm{span}\{X_i = x_{i(1)}\ldots x_{i(d)}: \ i:[d] \to \Lambda\}}^{\|\cdot\|_2}.$$  (Here we use the convention  $L^2_0(\mathcal B) := \C1_\mathcal A$).  Since the joint $\ast$-distribution of $X$ is invariant under free complexification, it follows from Lemma \ref{lem_star_partitions} and the moment cumulant formula (\ref{eqn_free_cum_implicit}), that the subspaces $\{L^2_d(\mathcal B)\}_{d \in \N \cup \{0\}}$ are orthogonal.  

\begin{prop} \label{cc_OU_semigroup}
Consider the Ornstein-Uhlenbeck type semigroup $\{\Gamma_t\}_{t \ge 0} \subset B(L^2(\mathcal B))$ given by  
\begin{equation*}
\Gamma_t = \sum_{d =0}^\infty e^{-dt}P_d.
\end{equation*}
Then $\{\Gamma_t\}_{t \ge 0}$ is a contraction semigroup on $L^2(\mathcal B)$, furthermore for each $t \ge 0$, $\Gamma_t$ restricts to a unital complete contraction $\Gamma_t:\mathcal B \to \mathcal B$. 
\end{prop}

\begin{proof}
The first statement is immediate since $\{P_d\}_{d \ge 0}$ is an orthogonal family of projections on $L^2(\mathcal B)$ and $e^{-dt} \le 1$ for all $t,d \ge 0$.  When $t=0$ the second statement is also immediate since $\Gamma_0|_{\mathcal B} = id_{\mathcal B}$, so assume for the remainder that $t > 0.$

Let $\T$ denote the unit circle in the complex plane and let $z = id_\T$ denote the canonical unitary generator of the C$^\ast$-algebra $C(\T)$.  Then $z$ is a Haar unitary in the C$^\ast$-probability space $(C(\T), \psi)$, where $\psi$ denotes integration with respect to normalized Haar measure on $\T$.  Since the joint $\ast$-distribution of $X$ is invariant under free complexification, there exists a state-preserving injective $\ast$-homomorphism of C$^\ast$-probability spaces \begin{eqnarray*} 
\pi:\big(C^*\langle 1_\mathcal A, X \rangle,  \varphi \big) &\to& \big(C(\T) *_{red} C^*\langle 1_\mathcal A, X \rangle, \psi*\varphi \big), \\
 \pi(x_r)  &=& zx_r,
\end{eqnarray*}
where $C^*\langle 1_\mathcal A, X \rangle \subset \mathcal A$ is the unital C$^\ast$-algebra generated by $X$.  Let $\{\rho_{t}\}_{t > 0}$ denote the Poisson convolution semigroup on $\T$, with convolution kernel given by the probability density $P_t(e^{i\theta}) = \frac{1- e^{-2t}}{1-2e^{-t}cos(\theta) + e^{-2t}}$ for $\theta \in [0,2\pi).$ It is well known that $\rho_{t}$ acts on $C(\T)$ by the formula 
\begin{eqnarray*} \rho_t(z^{n}) = P_t*(z^n) = e^{-|n|t}z^n, && (n \in \Z). 
\end{eqnarray*} 
Furthermore, since $P_t$ is a positive kernel for all $t>0$, $\rho_t$ is completely positive on $C(\T)$.  On the other hand, for any $d \in \N$ and $i:[d] \to \Lambda$, we have the identity \begin{eqnarray*}
\pi(\Gamma_t (x_{i(1)}x_{i(2)}\ldots x_{i(d)})) &=& \pi(e^{-dt} x_{i(1)}x_{i(2)}\ldots x_{i(d)}) \\
&=& e^{-dt} zx_{i(1)}zx_{i(2)}\ldots zx_{i(d)} \\
&=& (\rho_tz)x_{i(1)}(\rho_tz)x_{i(2)} \ldots (\rho_tz)x_{i(d)} \\
&=& (\rho_t*id)(zx_{i(1)}zx_{i(2)}\ldots zx_{i(d)})  \\
&=&  (\rho_t*id) \pi(x_{i(1)}x_{i(2)}\ldots x_{i(d)}),
\end{eqnarray*}
where $$\rho_t*id :C(\T) *_{red} C^*\langle 1_\mathcal A, X \rangle \to C(\T) *_{red} C^*\langle 1_\mathcal A, X \rangle$$ is the reduced free product of the completely positive, state-preserving maps $\rho_t$ and $id_{C^*\langle 1_\mathcal A, X \rangle}$ (which, by \cite{Ch}, is again completely positive and state preserving).  Therefore, by linearity and continuity, we have \begin{eqnarray*}
\pi \circ \Gamma_t = (\rho_t*id) \circ \pi|_{\mathcal B}, && (t > 0).
\end{eqnarray*}
Finally, since $\pi$ is a complete isometry, the completely bounded norm $\|\Gamma_t\|_{cb} = \|\Gamma_t\|_{CB(\mathcal B)}$ satisfies 
\begin{eqnarray*}
\|\Gamma_t\|_{cb} = \|\pi \circ \Gamma_t\|_{cb} = \|(\rho_t*id) \circ \pi|_{\mathcal B}\|_{cb} \le \|\rho_{t}*id\|_{cb} = 1, && (t >0).
\end{eqnarray*}
\end{proof}

\begin{rem}
It is easy to see that $(\rho_t*id)\pi(C^*\langle 1_\mathcal A, X \rangle) \subset \pi(C^*\langle 1_\mathcal A, X \rangle )$.  Therefore $\{\pi^{-1} \circ (\rho_t*id) \circ \pi\}_{t >0} \subset CB(C^*\langle 1_\mathcal A, X \rangle)$ is a $\varphi$-preserving completely positive extension of the semigroup $\{\Gamma_t\}_{t > 0}$ defined on $\mathcal B$.  This extension problem has been previously considered in the context of \textit{free} R-diagonal families in \cite{Ke}. 
\end{rem}

\begin{rem} We use the name ``Ornstein-Uhlenbeck'' in the previous proposition because of the fact that when $X$ is a free circular system, $\Gamma_t$ actually \textit{is} the free Ornstein-Uhlenbeck semigroup \cite{KeSp}. 
\end{rem}

Now suppose that $X = \{x_{r}\}_{r \in \Lambda} \subset (\mathcal A, \varphi)$ satisfies the hypotheses of Theorem \ref{main_theorem_tuples} or \ref{main_theorem_arrays}.  Let $\F_\Lambda^+$ denote the free semigroup with $|\Lambda|$ generators $\{g_{r}\}_{r \in \Lambda}$, and denote by $W_d \subset \F_\Lambda^+$ the set of words of length $d$ in $\F_\Lambda^+$.  Given $i:[d] \to \Lambda$, recall that $X_i:= x_{i(1)}x_{i(2)}\ldots x_{i(d)}$ and similarly put $g_i := g_{i(1)}g_{i(2)}\ldots g_{i(d)} \in \F_\Lambda^+.$  From Remark \ref{rem_orthogonality}, it follows that the map $g_i \mapsto X_i$, identifies $\ell^2(\F_\Lambda^+)$ with $L^2(\mathcal B)$, and consequently any $\psi \in \ell^\infty(\F_\Lambda^+)$ defines  a multiplication operator $M_\psi \in B(L^2(\mathcal B))$ given by
\begin{eqnarray*}
M_\psi (X_i) = \psi(g_i)X_i, && (i:[d] \to \Lambda).
\end{eqnarray*}
Note that $\|M_\psi\|_{B(L^2(\mathcal B))} = \|\psi\|_{\ell^\infty(\F_\Lambda^+)}$.  The next lemma says that if $\psi \in \ell^\infty(\F_\Lambda^+)$ decays sufficiently rapidly, then $M_\psi(L^2(\mathcal B)) \subseteq \mathcal B$. 

\begin{lem} \label{lem_mult_estimate}
Let $\psi \in \ell^\infty(\F_\Lambda^+)$ be such that $$K(\psi) := \sup_{d \ge 0}(d+1)^{3/2}\|\psi|_{W_d}\|_\infty < \infty.$$  Then $M_\psi(L^2(\mathcal B)) \subseteq \mathcal B$ and $\|M_\psi\|_{B(L^2(\mathcal B), \mathcal B)} \le C_\mathcal B K(\psi)$, where $C_\mathcal B$ is a constant only depending on $\mathcal B$. 
\end{lem}

\begin{proof}
Let $\psi \in \ell^\infty(\F_\Lambda^+)$ satisfy the above hypothesis.  Then, from our strong Haagerup inequality for $\mathcal B$ (Theorem \ref{main_theorem_tuples} or \ref{main_theorem_arrays}), there is a constant $c_\mathcal B > 0$ such that
\begin{eqnarray*}
\|T\|_{\mathcal B} \le c_\mathcal B \sqrt{d+1}\|T\|_{L^2(\mathcal B)}, && (T \in L^2_d(\mathcal B)).
\end{eqnarray*}
Put $C_\mathcal B = c_\mathcal B  \Big( \sum_{d=0}^\infty \frac{1}{(d+1)^2} \Big)^{1/2}$.  Then for any $T \in L^2(\mathcal B)$ we have (noting that $P_d M_\psi = M_\psi P_d$ for all $d \ge 0$)  
\begin{eqnarray*}
\|M_\psi T\|_{\mathcal B} &=& \Big\|\sum_{d=0}^\infty P_dM_\psi T\Big\|_{\mathcal B} \le \sum_{d=0}^\infty \|M_\psi P_dT\|_{\mathcal B} \\
&\le& c_\mathcal B \sum_{d=0}^\infty \sqrt{d+1} \|M_\psi P_dT\|_{L^2(\mathcal B)} \\
&\le& c_\mathcal B \sum_{d=0}^\infty \sqrt{d+1} \|\psi|_{W_d}\|_\infty \|P_dT\|_{L^2(\mathcal B)} \\
&\le& c_\mathcal B \sum_{d=0}^\infty \sqrt{d+1}\frac{K(\psi)}{(d+1)^{3/2}}\|P_dT\|_{L^2(\mathcal B)} \\
&\le& c_\mathcal B K(\psi) \Big( \sum_{d=0}^\infty \frac{1}{(d+1)^2} \Big)^{1/2} \Big(\sum_{d=0}^\infty\|P_dT\|_{L^2(\mathcal B)}^2\Big)^{1/2} \\
&=& C_\mathcal B K(\psi) \|T\|_{L^2(\mathcal B)} .
\end{eqnarray*}
Hence $M_\psi L^2(\mathcal B) \subseteq \mathcal B$ and $\|M_\psi\|_{B(L^2(\mathcal B), \mathcal B)} \le C_\mathcal B K(\psi).$
\end{proof}

We are now ready to prove the main theorem of this section.

\subsection{Proof of Theorem \ref{thm_MAP}}

Denote by $\chi_{W_d}$ the characteristic function of the set $W_d$.  For each $N \in \N$ and $t > 0$, define \begin{eqnarray*} \Gamma_{t,N}:\mathcal B \to \mathcal B, && \Gamma_{t,N} = \sum_{d=0}^N e^{-dt}P_d = \sum_{d=0}^N e^{-dt}M_{\chi_{W_d}}.
\end{eqnarray*}  Let $\psi_{t,N} = \sum_{d \ge N+1} e^{-dt}\chi_{W_d} \in \ell^\infty(\F_\Lambda^+)$, so that \begin{equation} \label{eqn} \Gamma_t - \Gamma_{t,N} = \sum_{d \ge N+1} e^{-dt}P_d = \sum_{d \ge N+1} e^{-dt}M_{\chi_{W_d}}  = M_{\psi_{t,N}}. \end{equation}    Since the inclusion $\mathcal B \hookrightarrow L^2(\mathcal B)$ is a contraction, we have  \begin{eqnarray*}
\|\Gamma_t - \Gamma_{t,N} \|_{B(\mathcal B)}&\le& \|\Gamma_t - \Gamma_{t,N}\|_{B(L^2(\mathcal B), \mathcal B)}  \le C_\mathcal B \sup_{d \ge N+1} (d+1)^{3/2}e^{-dt},
\end{eqnarray*} where the last inequality follows from (\ref{eqn}) and Lemma \ref{lem_mult_estimate}.  So for each $t > 0$, \begin{eqnarray} \label{eqn_lim1} \lim_{N \to \infty}\|\Gamma_{t,N} - \Gamma_t\|_{B(\mathcal B)} = 0, &\textrm{and}& \lim_{N \to \infty}\|\Gamma_{t,N}\|_{B(\mathcal B)} = \|\Gamma_t\|_{B(\mathcal B)} = 1.  
\end{eqnarray}  Let $Q_{t,N} = \|\Gamma_{t,N}\|^{-1}\Gamma_{t,N}$.  Then (\ref{eqn_lim1}) also gives \begin{eqnarray} \label{eqn_lim2}
\lim_{N\to \infty} \|Q_{t,N} - \Gamma_t\|_{B(\mathcal B)} = 0, && (t > 0).
\end{eqnarray}  

We now claim that the identity map $id_\mathcal B:\mathcal B \to \mathcal B$ is contained in the strong closure of the set of contractions $\{Q_{t,N}\}_{t > 0, N \in \N} \subset B(\mathcal B)$.  To prove this, first note that by (\ref{eqn_lim2}), $\{\Gamma_t\}_{t > 0}$ is contained in the strong closure of $\{Q_{t,N}\}_{t > 0, N \in \N}$.  Next, note that since $\lim_{t \to 0}e^{-dt} = 1$ for all $d >0$, we have $\lim_{t \to 0}\|\Gamma_t T - T\|_\mathcal A = 0$ for any polynomial $T \in \textrm{Alg}\langle 1_\mathcal A, X \rangle$.  Since  $\{\Gamma_t\}_{t > 0}$ is uniformly norm-bounded (by Proposition \ref{cc_OU_semigroup}), this limit is valid for all $T \in \mathcal B = \overline{\textrm{Alg}\langle 1_\mathcal A, X \rangle}^{\|\cdot\|_\mathcal A}$.  Therefore $id_\mathcal B$ is contained in $\overline{\{\Gamma_t\}_{t >0}}^{s.o.t.} \subset \overline{\{Q_{t,N}\}_{t>0, n \in \N}}^{s.o.t.}$, proving the claim.  

If $|\Lambda| < \infty$, then each of the maps $Q_{t,N}$ is finite rank, so the MAP for $\mathcal B$ follows from the fact that $\{Q_{t,N}\}_{t > 0,N \in \N}$ contains $id_\mathcal B$ in its strong closure.  If $|\Lambda| =  \infty$, then the contractions $\{Q_{t,N}\}_{t >0, N \in \N}$ constructed above are no longer finite rank.  Let $\mathcal F$ denote the collection of finite subsets of $\Lambda$.  For each $F \in \mathcal F$, let $X_F = \{x_r\}_{r \in F}$, and let $E_F$ be the unique $\varphi$-preserving conditional expectation from the von Neumann algebra $C^\ast \langle 1_\mathcal A, X \rangle ^{\prime \prime} \subseteq B(L^2(\mathcal A, \varphi))$ onto $C^*\langle 1_\mathcal A, X_F \rangle ^{\prime \prime} \subseteq B(L^2(\mathcal A, \varphi))$.  (Recall that on the $L^2$-level, $E|_\mathcal B$ is just the orthogonal projection from $L^2(\mathcal B)$ onto the Hilbert subspace generated by $\{1_\mathcal A, X_F\}$).  Put $Q_{t,N,F} = Q_{t,N} \circ E_F|_{\mathcal B}$.  Then $\{Q_{t,N,F} \}_{t > 0, N \in \N}$ is a family of finite rank contractions on $\mathcal B$, which, by the argument in the previous paragraph, contains $E_F|_{\mathcal B}$ in its strong closure.  For any $T \in \mathcal B$ we have $\lim_{F \in \mathcal F} \|E_F T - T\| = 0$, and therefore $\{Q_{t,N,F} \}_{t > 0, N \in \N, F \in \mathcal F}$ contains $id_\mathcal B$ in its strong closure, so $\mathcal B$ has the MAP when $|\Lambda| = \infty$ as well.

\section{Applications to Free Unitary Quantum Groups} \label{quantum_groups}

In this final section, we consider applications of our results to the reduced C$^\ast$-algebra associated to the \textit{free unitary quantum group} $U_n^+$ (of dimension $n$), introduced by Wang \cite{Wa}.  Let us briefly recall the definition of this quantum group: $U_n^+$ is the CMQG given by the pair $(A_u(n), U)$, where $A_u(n)$ is the universal C$^\ast$-algebra with generators $\{u_{rs}: 1 \le r,s \le n\}$ subject to the relations which make the matrices $U = [u_{rs}]_{1 \le r,s \le n}$ and $\overline{U} =[u_{rs}^*]_{1 \le r,s \le n}$ unitary in $M_n(A_u(n))$.  It is clear that for each $n \in \N$,  $U_n^+ = (A_u(n), U)$ satisfies Definition \ref{defn_cmqg}.  The \textit{free orthogonal quantum group} $O_n^+$ (of dimension $n$), also introduced in \cite{Wa}, will also be of use to us here.  $O_n^+$ is the CMQG given by the pair $(A_o(n), V)$, where $A_o(n) = A_u(n)/\langle u_{rs} = u_{rs}^* : 1 \le r,s \le n  \rangle$, and $V = [v_{rs}]_{1 \le r,s \le n}$ is the image of $U$ under the canonical quotient map $A_u(n) \to A_o(n)$.  For the rest of this section (and with a slight abuse of notation), we will also use the symbols $u_{rs}$ and $v_{rs}$ to  denote the canonical generators of $L^\infty(U_n^+)$ and $L^\infty(O_n^+)$, respectively.  We also note that these quantum groups are always unimodular, i.e. their Haar states are \textit{tracial} \cite{BaCo}.   

In recent years, the series $\{U_n^+\}_{n \in \N}$ and $\{O_n^+\}_{n \in N}$ have been intensively studied from both operator algebraic and probabilistic perspectives \cite{Ba, BaCo,  BaCuSp, BaCuSp2, BaCoZJ, VaVe, Ve}.  In particular, Vergnioux has proved a version of Haagerup's inequality for $O_n^+$ and $U_n^+$ \cite[Section 4.3]{Ve}.  Vergnioux's result can be formulated as follows (refer to \cite{Ti} for the unexplained terminologies below).   Let $\G = (A, U)$ be a
CMQG with fundamental corepresentation $U$, let $\widehat{\G}$ denote the collection of all unitary equivalence classes of irreducible finite dimensional unitary corepresentations of $\G$, and let $$\{U^\alpha = [u^\alpha_{ij}]_{1 \le i,j \le d_\alpha}\}_{\alpha \in \widehat{\G}},$$ be a complete family of representatives for $\widehat{\G}$.  Then by the Peter-Weyl Theorem (\cite{Wo}) $$\{u^{\alpha}_{ij}: \ \alpha \in \widehat{\G}, \ 1 \le i,j \le d_\alpha\} \subset L^\infty(\G)$$ forms an orthogonal basis for $L^2(\G)$.  Denote by $\mathcal C$ the category of equivalence classes of finite dimensional unitary corepresentations of $\G$, and let $S = \{\alpha_1, \ldots \alpha_s\} \subseteq \widehat{\G}$ be a generating set for $\mathcal C$ which is closed under conjugation of representations
and does not contain the trivial corepresentation $1_A$.  Let $\ell =\ell_S:\widehat{\G} \to \N$ be the ``word length'' function given by
$$l(\alpha) = \min\{k \in \N: \ \alpha \subseteq \alpha_{i(1)} \boxtimes \ldots \boxtimes \alpha_{i(k)}, \alpha_{i(j)} \in S\},$$ and
define $$L^2_d(\G):= \overline{\textrm{span}\{u^\alpha_{ij}:
\ell(\alpha) = d\}}^{\|\cdot\|_2} \subset L^2(\G).$$  Then we have the following definition.
\begin{defn} (\cite{Ve}) $\G$ has \textit{the property of rapid decay (property
RD)} (with respect to $\ell = \ell_S:\widehat{\G} \to \N$) if there
exists a polynomial $P \in \R_+[x]$ such that $$\|T\|_{L^\infty(\G)}
\le P(d) \|T\|_{L^2(\G)}$$ for all $T \in L^2_d(\G).$
\end{defn} 
For $O_n^+$, there is a natural labeling $\widehat{O_n^+} = \{V^{(k)}\}_{k \in \N\cup \{0\}}$ such that $1 = V^{(0)}$ and $V = V^{(1)}$.  For $U_n^+$ there is a natural labeling $\widehat{U_n^+} = \{U^g\}_{g \in \F_2^+}$ such that $1 = U^e$, $U = U^{g_1}$, and $\overline{U} = U^{g_2}$, where $g_1,g_2$ are the generators of $\F_2^+$.  Taking $S_O = \{V\}$ as a generating set for $\widehat{O_n^+}$, and $S_U = \{U,\overline{U}\}$ as generating set for $\widehat{U_n^+}$, the corresponding length functions $\ell_{S_O}, \ell_{S_U}$ are identified with the natural length functions on $\N \cup \{0\}$ and $\F_2^+$, respectively \cite[Section 4.3]{Ve}.  We collect here the main result from \cite[Section 4.3]{Ve}.

\begin{thm} \label{thm_Ve_result}
For each $n \in  \N$, there exist positive constants $C_n, D_n > 0$ such that $O_n^+$ has property RD with $P(x) = C_n(x+1)$, and $U_n^+$ has property RD with $P(x) = D_n(x+1)$.
\end{thm}  
Viewing $L^\infty(U_n^+)$ as a non-cocommutative analogue of $L(\F_n) = C^*_\lambda(\F_n)^{\prime \prime}$, Theorem \ref{thm_Ve_result} can be regarded as the non-cocommutative analogue of Haagerup's classical inequality (Theorem \ref{thm_Haagerup_ineq}) for $L(\F_n)$.  It is interesting to note that the order of growth (with respect to $d$) in Theorems \ref{thm_Haagerup_ineq} and \ref{thm_Ve_result} is the same.  Using our results from the previous sections we can also obtain the following non-cocommutative analogue of Theorem \ref{thm1_KeSp}, which improves on the growth of the bounds in Theorem \ref{thm_Ve_result}. 

\begin{thm} \label{thm_free_unitary}
Let $\mathcal B_n \subset L^\infty(U_n^+)$ be the norm-closed, non-self-adjoint unital subalgebra generated by the coefficients $\{u_{rs}\}_{1 \le r,s \le n}$ of the fundamental corepresentation of $U_n^+$ (and not their adjoints).  Then for any $d \in \N$, $p \in 2\N \cup \{\infty\}$, and any $T \in L_d^2(U_n^+) \cap \mathcal B_n$, we have $$\|T\|_{L^2(U_n^+)} \le \|T\|_{L^p(U_n^+)} \le  4^6 \cdot (3e)^3\sqrt{e} \sqrt{d+1}
\|T\|_{L^2(U_n^+)}.$$  
\end{thm}

\begin{proof}
Let $h:L^\infty(U_n^+) \to \C$ denote the Haar state.  The bi-invariance of $h$  with respect to $\Delta$ (equation (\ref{eqn_Haar})) implies that the joint $\ast$-distribution of the array $\{u_{rs}\}_{1\le r,s \le n} \subset (L^\infty(U_n^+), h)$ is $U_n^+$-bi-invariant.  By comparing the defining relations for $H_n^+$ and $U_n^+$, it is clear that $H_n^+$ is a quantum subgroup of $U_n^+$.  So by Remark \ref{rem_quantum_subgroups}, $\{u_{rs}\}_{1\le r,s \le n}$ has an $H_n^+$-bi-invariant joint $\ast$-distribution.  

In \cite{Ba} (see also \cite[Theorem 9.2]{BaCo}), it was shown that $\{u_{rs}\}_{1\le r,s \le n}$ is the free complexification of $\{v_{rs}\}_{1\le r,s \le n}$.  In particular, the joint $\ast$-distribution of $\{u_{rs}\}_{1\le r,s \le n}$ is invariant under free complexification by Remark \ref{rem_free_compl}.  Therefore the array $\{u_{rs}\}_{1\le r,s \le n} \subset (L^\infty(U_n^+),h)$ satisfies the hypotheses of Theorem \ref{main_theorem_arrays}.  Since the linear span of the homogeneous polynomials of degree $d$ in the variables $\{u_{rs}\}_{1\le r,s \le n} \subset L^\infty(U_n^+)$ is precisely $L^2_d(U_n^+) \cap \mathcal B_n$ \cite[Theorem 1]{Ba}, Theorem \ref{main_theorem_arrays} gives $$\|T\|_{L^2(U_n^+)} \le \|T\|_{L^p(U_n^+)} \le  4^5 \cdot (3e)^3\sqrt{e} \frac{\|u_{11}\|_p^2}{\|u_{11}\|_2^2}\sqrt{d+1} \|T\|_{L^2(U_n^+)},$$ for any $T \in L^2_d(U_n^+) \cap \mathcal B_n.$  To complete the proof, we use a result of Banica, Collins, and Zinn-Justin \cite[Theorem 5.3]{BaCoZJ}, which says that the spectral measure of any generator $v_{rs} \in L^\infty(O_n^+)$ with respect to the Haar state has support equal to $\Big[\frac{-2}{\sqrt{n+2}},\frac{2}{\sqrt{n+2}}\Big]$.  Since also $\|v_{11}\|_2^2 = n^{-1}$, we get $$\frac{\|u_{11}\|_p^2}{\|u_{11}\|_2^2} \le \frac{\|u_{11}\|_\infty^2}{\|u_{11}\|_2^2} = \frac{\|zv_{11}\|_\infty^2}{\|zv_{11}\|_2^2} = \frac{\|v_{11}\|_\infty^2}{\|v_{11}\|_2^2} = \frac{4n}{n+2} \le 4.$$
\end{proof}

\begin{rem} \label{rem_optimality}
We can show that the order of the growth of the bounds in both Theorems \ref{thm_Ve_result} and \ref{thm_free_unitary} is optimal.  Let $\chi_O = (Tr \otimes id)V \in L^\infty(O_n^+)$ and $\chi_U  = (Tr \otimes id)U \in L^\infty(U_n^+)$ denote the fundamental characters of $O_n^+$and $U_n^+$, respectively.  Then $\chi_O$ is a standard semicircular random variable in $(L^\infty(O_n^+), h)$, and $\chi_U$ is a standard circular random variable in $(L^\infty(U_n^+), h)$ \cite{Ba, BaCo}.  Let $\{T_d\}_{d  \in \N\cup \{0\}}$ denote the Chebyshev II polynomials determined by the initial conditions $T_0(x) = 1, \ T_1(x) = x$ and the recursion 
\begin{eqnarray} \label{eqn_recursion}
xT_d(x) = T_{d+1}(x) + T_{d-1}(x), && (d \ge 1).
\end{eqnarray}
It is well known that these polynomials are orthonormal for the standard semicircular law.  Since both $\chi_o$ and $\sqrt{2}\textrm{Re}\chi_U$ are standard semicircular, functional calculus gives $$\|T_d(\chi_O)\|_{L^\infty(O_n^+)} = \|T_d(\sqrt{2}\textrm{Re}\chi_U)\|_{L^\infty(U_n^+)} = \sup_{t \in [-2,2]}|T_d(t)| = d+1. $$ 
On the other hand, a simple inductive argument on $d \in \N$ using the orthogonality of the families $\{T_d(\chi_O)\}_d$ and $\{T_d(\sqrt{2}\textrm{Re}\chi_U)\}_d$ and (\ref{eqn_recursion}) shows that $T_d(\chi_O) \in L^2_d(O_n^+)$ and $T_d(\sqrt{2}\textrm{Re}\chi_U) \in L^2_d(U_n^+)$.  So the growth rate of $O(d+1)$ given by Theorem \ref{thm_Ve_result} is actually obtained.  

Similarly, since $\chi_{U}$ is standard circular, we have $\chi_U^d \in L^2_d(U_n^+) \cap \mathcal B_n$, $\|\chi_{U}^d\|_{L^2(U_n^+)} = 1$, and by \cite[Corollary 3.2]{KeSp} 
\begin{equation*}
\|\chi_{U}^d\|_{L^\infty(U_n^+)} = (1 + 1/d)^{d/2}\sqrt{d+1}  \sim^{d \to \infty} \sqrt{e(d+1)}.
\end{equation*}
So the growth rate of $O(\sqrt{d+1})$ given by Theorem \ref{thm_free_unitary} is actually obtained.  Of course, the universal constant $4^6(3e)^3\sqrt{e}$ given by Theorem \ref{thm_free_unitary} can probably be greatly improved.     
\end{rem} 

Denote by $C(U_n^+) \subset L^\infty(U_n^+)$ the GNS representation of $A_u(n)$ with respect to the Haar state.  In \cite{Ba}, it is shown that $C(U_n^+)$ is non-nuclear and simple, $L^\infty(U_n^+)$ is a non-injective $II_1$-factor for all $n \in \N$, and $L^\infty(U_2^+) \cong L(\F_2) = C^*_\lambda(\F_2)^{\prime\prime}$.  It is an interesting open question whether $L^\infty(U_n^+)$ is a free group factor for all $n \in \N$?  It would also be interesting to know what other properties the algebras $C(U_n^+)$ and $L^\infty(U_n^+)$ share with the free group algebras.  For example, does $C(U_n^+)$ always have the MAP?  A related question on the von Neumann level is: does $L^\infty(U_n^+)$ always have the Haagerup approximation property? See \cite{Jo1} for the definition of this property.  We note that the answer to these last two questions would be ``yes'' if one could show that the exponentiated length function $e^{-\ell t}:= \bigoplus_{d \ge 0} e^{-dt} id_{L^2_d(U_n^+)} \in B(L^2(U_n^+))$ restricted to a unital completely positive map on $L^\infty(U_n^+)$, for each $t >0$.  Unfortunately this is not true \cite[Section 1]{Ve}, and these questions remain open.  On the positive side, a direct application of Theorem \ref{thm_MAP} gives the following partial result concerning the MAP.  

\begin{thm} \label{thm_MAP_QG}
Let $\mathcal B_n \subset C(U_n^+) \subset L^\infty(U_n^+)$ be the unital non-self-adjoint operator algebra generated by the coefficients of the fundamental corepresentation of $U_n^+$.  Then $\mathcal B_n$ has the MAP for all $n \in \N$.
\end{thm}   

In \cite{Ba2} the notion of free complexification of a compact matrix quantum group was introduced.  Let $\G = (A, U = [u_{rs}]_{1 \le r,s \le n})$ be a CMQG, and let $z = id_\T$ be the canonical generator of $C(\T)$.  The \textit{free complexification} of $\G$ is the CMQG $\tilde{\G} = (\tilde{A}, \tilde{U})$, given by $\tilde{U} = [zu_{rs}]_{1 \le r,s \le n} \in M_n(C(\T) * A)$, and $\tilde{A} = C^*\langle zu_{rs}: 1 \le r,s \le n \rangle \subseteq C(\T)*A$.  The Haar state on $\tilde{A}$ is the restriction to $\tilde{A}$ of the free product of the Haar states on $C(\T)$ and $A$. 
Using this notion, we can construct more families of random variables which satisfy the hypotheses of Theorems \ref{main_theorem_tuples} and \ref{main_theorem_arrays}.  Indeed, if $\G = (A,U)$ is \textit{any} CMQG containing $H_n^+$ as a quantum subgroup, then $\tilde{\G} = (\tilde{A},\tilde{U})$ also contains $H_n^+$ as a quantum subgroup \cite{Ba2}, so the array of variables given by $\tilde{U}$ is $H_n^+$-bi-invariant, and invariant under free complexification.  The special case $U_n^+ = \widetilde{O_n^+}$ was considered above.
  The fact that these variables are not $\ast$-free follows from an argument similar to \cite[Section 4.9]{Cu}.


\begin{thebibliography}{99}

\bibitem{Ba} Banica, T.: Le groupe quantique compact libre $U(n)$. Comm. Math. Phys. 190, 143--172 (1997).

\bibitem{Ba2} Banica, T.: A note on free quantum groups. Ann. Math. Blaise Pascal. 15, 135--146 (2008).

\bibitem{BaBiCo} Banica, T., Bichon, J., Collins, B.: The hyperoctahedral quantum group.  J. Ramanujan Math.
Soc. 22, 345--384 (2007).

\bibitem{BaCo} Banica, T., Collins, B.:  Integration over compact quantum groups.  Publ. Res. Inst. Math. Sci. 43, 277--302 (2007).

\bibitem{BaCuSp0}   Banica, T., Curran, S., Speicher, R.:  Classification results for easy quantum groups.  Pacific J. Math. 247, 1-26 (2010). 

\bibitem{BaCuSp} Banica, T., Curran S., Speicher, R.:  De Finetti theorems for easy quantum groups.  Ann. Prob.  To appear. 

\bibitem{BaCuSp2} Banica, T., Curran, S., Speicher, R.:  Stochastic aspects of easy quantum groups. Probab. Theory Related Fields. To appear.

\bibitem{BaCoZJ} Banica, T.,  Collins, B., Zinn-Justin, P.:  Spectral analysis of the free orthogonal matrix.  Int. Math. Res. Notices (2009). DOI: $10.1093/imrn/rnp054$

\bibitem{BaSp} Banica , T., Speicher, R.: Liberation of orthogonal Lie groups. Adv. Math. 222, 1461--1501 (2009).

\bibitem{Bo} Bo\.zejko, M.:  Remark on Herz-Schur multipliers on free groups. Math. Ann. 258, 11-15 (1981).

\bibitem{Ch} Choda, M.:  Reduced free products of completely positive maps and entropy for free product of automorphisms. Publ. Res. Inst. Math. Sci. 32, 371--382 (1996).

\bibitem{CoHa}  Cowling, M., Haagerup, U.: Completely bounded multipliers of the Fourier algebra of a simple Lie group of real rank one. Invent. Math. 96, 507-549 (1989).

\bibitem{Cu} Curran, S.: Quantum Rotatability. Trans. Amer. Math. Soc. To appear.

\bibitem{CuSp} Curran, S., Speicher, R.:  Asymptotic infinitesimal freeness with amalgamation for Haar quantum unitary random matrices. Comm. Math. Phys.  To appear.

\bibitem{dCHa} de Canniere, J., Haagerup, U.:  Multipliers of the Fourier algebras of some simple Lie Groups and their discrete subgroups.  Amer. Journ. Math. 107, 455--500 (1985).

\bibitem{dS} de la Salle, M.:  Strong Haagerup inequalities with operator coefficients.  J. Funct. Anal.  257, 3968-4002 (2009).

\bibitem{Ed} Edelman, P.: Chain enumeration and non-crossing partitions.  Discrete Math. 31, 171--180 (1980).

\bibitem{Ha} Haagerup, U.: An example of a nonnuclear C$^{\ast}$-algebra, which has the metric approximation property.  Invent. Math.  50, 279--293 (1978/79).

\bibitem{Jo} Jolissaint, P.: K-theory of reduced C$^\ast$-algebras and rapidly decreasing functions on groups.  K-Theory. 2, 723--735 (1989).

\bibitem{Jo1} Jolissaint, P.:  The Haagerup approximation property for finite von Neumann algebras.  J. Operator Theory. 48, 549--571 (2002). 

\bibitem{Ke} Kemp, T.:  R-diagonal dilation semigroups. Math. Z. 264, 111--136 (2010).

\bibitem{KeSp} Kemp, T. and R. Speicher, {\it Strong Haagerup inequalities for free $R$-diagonal elements}.  J. Funct. Anal.  251, 141--173 (2007).

\bibitem{KoSp} K\"ostler, C., Speicher, R.:  A noncommutative de Finetti theorem: invariance under quantum permu-
tations is equivalent to freeness with amalgamation. Comm. Math. Phys. 291, 473--490 (2009).

\bibitem{Kr} Kreweras, G.: Sur les partitions non-croisses d'un cycle. Discrete Math. 1, 333--350 (1972).

\bibitem{La1} Lafforgue, V.:  A proof of property (RD) for cocompact lattices of $SL(3,\R)$ and $SL(3,\C)$.  J. Lie Theory.  10, 255--267 (2000).

\bibitem{La2} Lafforgue, V.:  K-th\'eorie bivariante pour les alg\`ebres de Banach et conjecture de Baum-Connes. Invent. Math. 149, 1--95 (2002).

\bibitem{Lar} Larsen, F.: Powers of R-diagonal elements. J. Operator Theory.  47, 197--212 (2002).

\bibitem{NiSp} Nica, A., Speicher, R.: Lectures on the combinatorics of free probability.  LMS Lecture Notes Series 335, Cambridge Univ. Press,
(2006).

\bibitem{NiSp2} Nica, A., Speicher, R.: R-diagonal pairs - a common approach to Haar unitaries and circular elements.  Fields Inst. Commun. 12, 149--188 (1997).

\bibitem{Or} Oravecz, F.: On the powers of Voiculescu’s circular element.  Studia Math. 145, 85--95 (2001).

\bibitem{Sp} Speicher, R.:  Multiplicative functions on the lattice of non-crossing partitions and free convolution. Math. Annalen.  298, 611--628 (1994).

\bibitem{Ti} Timmerman, T.:  An invitation to quantum groups and duality. EMS Textbooks in Mathematics, Zurich (2008).

\bibitem{VaVe} Vaes, S.,  Vergnioux, R.: The boundary of universal discrete quantum groups, exactness, and factoriality. Duke Math. J. 140, 35--84 (2007).

\bibitem{Ve} Vergnioux, R.: The property of rapid decay for discrete quantum groups.  J. Operator Theory. 57, 303--324 (2007).

\bibitem{Wa} Wang, S.:  Free products of compact quantum groups.  Comm. Math. Phys. 167, 671--692 (1995).

\bibitem{Wo} Woronowicz, S.: Compact matrix pseudogroups.  Comm. Math. Phys. 111, 613--665 (1987).

\end{thebibliography}
\end{document}